\documentclass[a4paper,10pt]{scrartcl}
\usepackage[utf8]{inputenc}
\usepackage{csquotes}
                  \usepackage{hyperref}
                  \usepackage{amsthm}
                  \usepackage{physics}
                  \usepackage{svg}
                  \newtheorem{theorem}{Theorem}
                        
                  \newtheorem{lemma}[theorem]{Lemma}
                  \newtheorem{proposition}[theorem]{Proposition}
                  \theoremstyle{definition}
                  \newtheorem{remark}[theorem]{Remark}
                  \newtheorem{definition}[theorem]{Definition}
                  \newtheorem{nomenclature}[theorem]{Nomenclature}
                 
                 \usepackage{bm}

\usepackage[russian, english]{babel}

\usepackage[T1]{fontenc}
\usepackage{graphicx}
\usepackage{ mathrsfs }
\usepackage{longtable}
\usepackage{wrapfig}
\usepackage{rotating}
\usepackage{amsmath}
\usepackage{amssymb}
\usepackage{todonotes}
\usepackage{capt-of}
\usepackage{hyperref}

\newcommand{\R}{\mathbb{R}}

\renewcommand{\exp}{\mathrm{e}}

\newcommand{\bDelta}{\bm{\Delta}}

\newcommand{\laplaciano}{\text{{$\ell$}-Laplacian}}

\newcommand{\XS}{\widehat{X}}
\newcommand{\XSunnorm}{{X}}
\newcommand{\filo}{\ell}
\newcommand{\surf}{S}

\DeclareMathOperator{\Hess}{Hess}
\newcommand{\linspan}{\operatorname{span}}
\renewcommand{\div}{\operatorname{div}}
\newcommand{\dir}{\begin{otherlanguage*}{russian}\textrm{\CYRD}\end{otherlanguage*}}
\newcommand{\fol}{\mathscr{F}}

\author{Riccardo Adami\thanks{Politecnico di Torino, Dipartimento di Scienze Matematiche “G.L. Lagrange”,
Corso Duca degli Abruzzi, 24, 10129, Torino, Italy. \texttt{riccardo.adami@polito.it}} 
    \and 
    Ugo Boscain\thanks{CNRS, Sorbonne Université, Inria, Université de Paris, Laboratoire Jacques-Louis Lions, Paris, France. \texttt{ugo.boscain@sorbonne-universite.fr}}
    \and 
    Dario Prandi\thanks{Université Paris-Saclay, CNRS, CentraleSupélec, Laboratoire des signaux et systèmes, 91190, Gif-sur-Yvette, France. \texttt{dario.prandi@centralesupelec.fr}}
    \and 
    Lucia Tessarolo\thanks{Sorbonne Université, Inria, Université de Paris, Laboratoire Jacques-Louis Lions, Paris, France. \texttt{lucia.tessarolo@sorbonne-universite.fr}}}
\date{\today}
\title{Schrödinger evolution on surfaces in 3D contact sub-Riemannian manifolds}
\hypersetup{
 pdfauthor={Riccardo Adami, Ugo Boscain, and Dario Prandi},
 pdflang={English}}
\usepackage[maxbibnames=4, giveninits=true, backend=biber]{biblatex}
\begin{document}

\maketitle

\begin{abstract}
    Let $M$ be a 3-dimensional contact sub-Riemannian manifold and $S$ a surface embedded in $M$.
    Such a surface inherits a field of directions that becomes singular at characteristic points. The integral curves of such field define a characteristic foliation $\mathscr{F}$. 
 In this paper we study the Schrödinger evolution of a particle constrained on $\mathscr{F}$. In particular, we relate the self-adjointness of the Schr\"odinger operator with a geometric invariant of the foliation. We then classify a special family of its self-adjoint extensions: those that yield disjoint dynamics.

\end{abstract}

\tableofcontents

\section{Introduction}
\label{sec: intro}

Consider a 3-dimensional contact sub-Riemannian manifold defined as a triple $(M, D, g)$, where $M$ is a connected smooth manifold, $D \subset TM$ is a smooth vector distribution satisfying $\dim D_q = 2$ and $\dim (D + [D, D])_q = 3$ for all $q \in M$, and $g$ is a Riemannian metric defined on $D$ (for more details see  \cite{agrachevComprehensive2019a}).

Let $\surf$ be a smooth surface embedded in $M$. We denote by $C(S)$ the set of characteristic points, i.e. the points at which the distribution $D$ is tangent to the surface $S$. We assume that this set consists of isolated points, which is actually a generic condition (see \cite{barilariInduced2022}).
We can define the map $\dir :  \surf \setminus C(S) \ni q \mapsto D_q \cap T_q \surf$, which specifies a field of directions on $\surf \setminus C(S)$. The integral curves of this field form the \emph{characteristic foliation} $\fol$ of $\surf \setminus C(S)$, consisting of $1$-dimensional leaves $\ell$. Such foliation can be visualized as a graph with vertices that connect a possible continuum of edges. 
Locally, this foliation can be defined by the integral curves of a vector field on $S$, called the \emph{characteristic vector field}, 
whose set of singular points is $C(S)$. Henceforth, to simplify the discussion, we assume the characteristic vector field to be globally defined and complete.

Each leaf $\ell$ of the characteristic foliation is endowed with a locally Euclidean metric, inherited from the Riemannian metric on $D$.
Hence, $\ell$ is isometric either to a circle or to an open (finite or infinite) segment $(a_0, a_1)$, $-\infty\le a_0 < a_1\le \infty$ and in this case we write $\filo\sim (a_0,a_1)$. 
If $a_0$ (resp. $a_1$) is finite and the corresponding limit point of $\ell$ belongs to $S$, then such point is a characteristic point (see Proposition 1.3 of  \cite{barilariInduced2022}).

Recall that any contact sub-Riemannian manifold is naturally equipped with a volume $\mathcal{P}$, called the \emph{Popp volume}\footnote{To simplify the exposition, we henceforth assume that $\mathcal{P}$ is given by a global volume form (see Section~\ref{sec: volume}).}.
Although $\surf$ has no sub-Riemannian structure, it inherits a volume form $\mu$ defined on $S\setminus C(S)$ via the contraction of $\mathcal{P}$ along the horizontal normal $N$ to $\surf$. 
This volume form is globally defined if $S$ is orientable (a hypothesis that we assume all along the paper) and reads
\begin{equation}
    \label{hor unit norm}
    \mu(\cdot,\cdot)=\mathcal{P}(N,\cdot,\cdot),
\end{equation}
where, for each $q\in S\setminus C(S)$, $N_q\in D_q$ is defined up to a sign by the following property:
\[
 g_q(N_q,w)=0 \ \ \forall \, w\in T_qM\cap D_q.
\]
It follows that the volume form $\mu$ is non-vanishing on $S\setminus C(S)$ and can be extended to the whole $S$ by continuity (see Proposition~\ref{prop: mu va a zero}).
This allows to define a Laplace-Beltrami-like operator on $\ell$ ($\laplaciano$ for short), via the expression
\begin{equation}
\label{lapl div grad}
    \Delta_\ell u = \div_\mu \nabla_\ell u.
\end{equation}
Here $\nabla_\ell$ is the Riemannian gradient associated with the locally Euclidean structure on $\ell$, and $\div_\mu$ is the divergence w.r.t. the surface measure $\mu$. 

The $\laplaciano$ $\Delta_\ell$ was first introduced in \cite{barilariStochastic2021} as the 
 limit of the classical Laplace-Beltrami operator on the surface $S$  as embedded in a family of Riemannian manifolds approximating $(M,D,g)$. The intrinsic expression of $\Delta_\ell$ \eqref{lapl div grad} was shown in \cite{Barilari_2023}.

As explained in Section \ref{sec: disi}, the measure $\mu$ disintegrates on the leaves of $\fol $ into a family of one dimensional measures $\mu_\ell$, $\ell\in \fol $. Each of such measures is unique up to a multiplicative constant which does not affect the divergence. Hence  
$\Delta_\ell  = \div_\mu \nabla_\ell =\div_{\mu_\ell} \nabla_\ell $ can be seen as an operator acting on functions of one variable (defined on $\ell$). As a consequence, the operator $\Delta_\ell$ cannot be hypoelliptic. By construction $\Delta_\ell$ is symmetric with respect to the measure $\mu_\ell$.

While the heat diffusion associated with $\Delta_\ell$ was studied in \cite{barilariStochastic2021}, in the present
paper we analyse the Schrödinger evolution 
\begin{equation}
    \label{eq:schrodinger}
    i\partial_t \psi_\ell = -\Delta_\ell \psi_\ell.
\end{equation}

For investigating such equation,
we focus on the issue of the self-adjointness of the operator $-\Delta_\ell$.
Self-adjointness is an essential feature for any operator representing a quantum observable quantity, for two reasons: first, the spectrum of any self-adjoint operator is real, 
allowing for the interpretation of its elements as the possible results of a suitable measurement; second, the evolution
resulting from \eqref{eq:schrodinger} is unitary, 
entailing the conservation of probability.
It is worth stressing that the self-adjointness of an operator depends not only on its action, but also on its domain, tipically characterized by suitable boundary conditions. In some
cases the natural domain is a self-adjointness domain, thus self-adjointness is not an issue. For instance, the standard Laplacian $-\Delta$ on
$L^2 (\R)$ turns out to be self-adjoint in its natural  domain $H^2 (\R)$. In other cases, the natural domain does not automatically provides self-adjointness: the standard Laplacian $- \Delta$ on $L^2 (\R^+)$ is not self-adjoint on the domain $H^2 (\R^+),$ but
requires an additional boundary condition at the origin, that can be of Dirichlet, Neumann, or Robin type.

The dynamical interpretation of the self-adjointness of $\Delta_\ell$ on $L^2(\ell,\mu_\ell)$ is transparent: owing to the conservation of probability, if $\Delta_\ell$ is self-adjoint, then the Schrödinger evolution \eqref{eq:schrodinger}
does not allow the solution to leave the leaf $\ell$. It appears here that the boundary conditions play a crucial role, in that they prevent the solution from escaping. Indeed a major part of our analysis
focuses on them.

In order to prove the self-adjointness of $\Delta_\ell$, we show that on a given leaf $\ell$, the equation \eqref{eq:schrodinger} is unitarily equivalent to a Schrödinger equation of the form 
\begin{equation}
    \label{eq:schrodinger}
    i\partial_t \psi_\ell = \left(
    -\partial_s^2+ V(s)\right)\psi_\ell \text{  in  }L^2(\ell,ds).
\end{equation}
Here, $s$ is the arc-length parameter along the leaf $\ell$ and $V$ is a potential which diverges at characteristic points $p$ of $\ell$. Hence, the problem of the self-adjointness of $\Delta_\ell$ can be reduced to the study of the self-adjointness w.r.t.~the standard Lebesgue volume $ds$ of the standard Laplacian in dimension one with a divergent potential.

For the sake of exposition, we first present our problem on an 
explicitly solvable model.

\begin{remark}
    In the study of 3-dimensional contact distributions from a topological perspective, embedded surfaces play a central role \cite{BennequinEQPfaff, GirouxConvexité91, GirouxBifurcat00}. Recently, related topics were explored in the case of 3-dimensional contact sub-Riemannian manifolds, i.e. contact manifolds in which the distribution is equipped with a metric.  See for instance \cite{DANIELLI2007292, DanielliMean12} for Carnot groups, \cite{barilariInduced2022, Eugenio-quantization} for generic structures, \cite{BaloghCorrectionHeisenberg20, BaloghHeisen17, VelosoGauss-Bonnet20} for Gauss-Bonnet theorems.
\end{remark}

\subsection{Example: Hyperbolic paraboloid in the Heisenberg case}
\label{hyper para}
As an example, let us consider the 3-dimensional contact sub-Riemannian  Heisenberg manifold $(\mathbb{H},D,g)$, $\mathbb{H}\cong \R^3$, $D=\linspan\{X_1,X_2\}$, where
$$
X_1=\partial_x - \frac y2 \partial_z
\qquad\text{and}\qquad
X_2 = \partial_y + \frac x2 \partial_z,
$$
and the metric $g$ on $D$ is chosen in order for $\{X_1,X_2\}$ to be an orthonormal frame.
The Popp volume on $\mathbb{H}$ is
$\mathcal{P} = dx\wedge dy\wedge dz$.

Consider now a hyperbolic paraboloid embedded in $\mathbb{H}$, i.e., $\surf=\{z=a x y\}$, where $a\ge 0$ is a real parameter.
When $a\ne \frac{1}{2}$ the surface $S$ has a unique characteristic point at the origin. The case $a=1/2$ is degenerate in the sense that there is a line of  characteristic points. In the following we consider only the case  
$a\neq1/2$.

The horizontal unit normal to the surface is
\[
N=\frac{X_1(u)}{\sqrt{X_1(u)^2+X_2(u)^2}}X_1+\frac{X_2(u)}{\sqrt{X_1(u)^2+X_2(u)^2}}X_2,
\]
where $u=z-axy$ and we made an arbitrary choice of sign. Notice that $N$ is not defined at the origin.

Using the unit normal, one can find the volume form \eqref{hor unit norm} inherited  on $S\setminus \{0\}$ from the Heisenberg structure. It reads
\begin{equation}
\mu=\iota_N(\mathcal{P})=\sqrt{\lambda_-^2 x^2+\lambda_+^2 y^2}\,dx\wedge dy \mbox{~~~ where ~~~}\lambda^\pm=\frac12\pm a.
\end{equation}

\noindent
The field of directions on $\surf\setminus\{(0,0,0)\}$ is 
\begin{equation}
    \dir(x,y,z)=\linspan \XSunnorm(x,y,z) 
    \quad\text{where}\quad
    \XSunnorm(x,y,z) = x\lambda^- 
    \partial_x+y\lambda^+
    \partial_y+\frac{xy}{2}\partial_z.
\label{eq:dir}
\end{equation}
It is easy to verify that the characteristic point at the origin is a proper node if $0< a<\frac12$, a saddle if $a>\frac12$, and a star node if $a=0$, in which case
the surface $\surf$ is the $(x,y)$ plane. In our point of view such types of singularity radically differ from one another, since in the case of a node there is a continuous of leaves arriving at the origin, while in the case of a saddle  only the four separatrices arrive there. This distinction  strongly impacts on any evolution on the foliation one aims at studying, in particular on that generated by the Schrödinger's equation.

The $x$ and the $y$ axes always belong to the characteristic foliation and moreover they coincide with  the eigenspaces at the origin of $D\XSunnorm$, where $\XSunnorm$ is the vector field defined in formula \eqref{eq:dir}. 

Let us focus on the expression of the operator $\Delta_\ell$, defined in \eqref{lapl div grad}, along the $x$ and $y$  axes. On any leaf $\ell$ of the foliation, the gradient $\nabla_\ell$ must be interpreted as $\nabla_\ell(\cdot)=\XS(\cdot)\XS=\frac{d}{ds}(\cdot)\XS,$ where 
\[\XS=\frac{\XSunnorm}{\norm{\XSunnorm}_g}\]
and $s$ is the arc-length parameter of $\ell$.
Then one obtains
\[
\Delta_x=\partial^2_x+\frac{1}{\lambda^-x}\partial_x
\qquad\text{and}\qquad
\Delta_y=\partial^2_y+\frac{1}{\lambda^+y}\partial_y.
\]
Such operators are symmetric w.r.t.~the measures $\mu_x = x^{1/\lambda^-}dx$ and $\mu_y = y^{1/\lambda^+}dy$, respectively. 
Actually, the measures that render $\Delta_x$ and $\Delta_y$ symmetric are disintegrations of $\mu$ along the two axes and are determined up to an irrelevant multiplicative constant.
More details on how  $\mu$ disintegrates in $\mu_x$ and $\mu_y$ can be found in Section~\ref{sec: disi}.

We stress that in the saddle case $a>1/2$ the measure $\mu_x$ explodes as $x\to 0$, even though $\mu_y$ goes to $0$ as $y\to 0$. In fact, even tough the measure $\mu$ goes to zero at characteristic points, the behaviour of the disintegration $\mu_\ell$ might depend on the leaf.

The unitary transformations $T_x(u)=x^{\frac{1}{2\lambda_-}}u$ and $T_y(u)=y^{\frac{1}{2\lambda_-}}u$ allow to rewrite the operators $\Delta_x$ and $\Delta_y$ as, respectively,
\[
L_x = \partial_x^2 - \frac{c_-}{x^2},
\quad\text{and}\quad
L_y = \partial_y^2 - \frac{c_+}{y^2}, 
\qquad\text{where}\quad
c_\pm = \frac{1-2\lambda^\pm}{4(\lambda^\pm)^2} = \mp \frac{2a}{(1\pm 2a)^2}
\]
which are symmetric with respect to the standard Lebesgue measure.
Notice that $-L_x$ and $-L_y$ can be written as
\begin{gather}
-L_x = -\partial_x^2 + V_x
\quad\text{where}\quad
V_x = \frac{2a}{(1- 2a)^2}\frac1{x^2}
\quad\text{is a repulsive potential,}
\\
-L_y = -\partial_y^2 +V_y
\qquad\text{where}\quad
V_y = -\frac{2a}{(1- 2a)^2}\frac1{y^2}
\quad\text{is an attractive potential}.
\end{gather}
As customary, we reduce the problem of self-adjointness to that of the essential self-adjointness, namely the existence of a unique self-adjoint extension.

It is well-known that the essential self-adjointness on $L^2(\R_+)$ of the 1-dimensional operator $L = -\partial_s^2 + cs^{-2}$, $s \in \R_+$, on the domain $\mathcal{C}^\infty_0(\R_+)$,\footnote{i.e., the space  of smooth functions supported on a compact subset of the interval $(0,+\infty)$.} is equivalent to the request $c\ge 3/4$ (see \cite[Theorem~X.10]{reedFourier2007}). We readily obtain the following.

\begin{proposition} 
    \label{prop:heisenberg-sa}
   Let $a\ge 0$ and $a\neq 1/2$. Consider $\Delta_x$ and $\Delta_y$ with domain $\mathcal{C}^\infty_0(\R_+)$.
    Then, the operator $\Delta_y$ is not essentially self-adjoint in $L^2(\R_+)$. Moreover,
    \begin{itemize}
        \item For $a =0$ the origin is of star node type and $\Delta_x$ is not essentially self-adjoint in $L^2(\R_+)$;
         \item for $0 < a < 1/6$ the  origin is of proper node type and $\Delta_x$ is not essentially self-adjoint in $L^2(\R_+)$;
        \item for $1/6 \le a < 1/2$ the  origin is of node type and $\Delta_x$ is essentially self-adjoint in $L^2(\R_+)$;
        \item for $1/2 < a \le 3/2$ the  origin is of saddle type and $\Delta_x$ is essentially self-adjoint in $L^2(\R_+)$;
        \item for $a > 3/2$ the origin is of saddle type and $\Delta_x$ is not essentially self-adjoint in $L^2(\R_+)$.
    \end{itemize}
\end{proposition}

\subsection{The general case and the Laplacian on the foliation}
\label{sec:intro-lapl-fol}

Proposition~\ref{prop:heisenberg-sa} highlights the richness of behaviours that can be encountered at characteristic points. Actually, in the general case of a surface in a 3-dimensional contact sub-Riemannian manifold, the self-adjointness of the $\laplaciano$ $\Delta_\ell$ with domain $\mathcal{C}^\infty_0(\ell)$\footnote{The choice of $\mathcal{C}^\infty_0(\ell)$ as a domain, i.e. the space of smooth functions that vanish in a neighbourhood of both endpoints of $\ell$, is quite standard in this type of problems.} on the leaf $\ell$ depends solely on the structure of the endpoints of $\ell$.
In Section~\ref{sec: ess-self}, following Weyl's limit-circle/limit-point criterion, we introduce the notion of essential self-adjointness at the endpoint $p$ of $\ell$ in such a way that $\Delta_\ell$ is essentially self-adjoint if and only if it is essentially self-adjoint at both endpoints of $\ell$.

More specifically, the self-adjointness of $\Delta_\ell\in\mathbb{R}$ at the characteristic point $p\in C(S)$ depends on the curvature-like invariant $\hat{K}_p$ introduced in \cite{barilariInduced2022} (see Section~\ref{sec:curvature}). 
Indeed, such an invariant determines the local structure of the foliation $\fol$ near $p$. 
Roughly speaking, this establishes if the characteristic vector field has a saddle, a star node, a proper node, or a focus singularity at $p$ (see Figure~\ref{fig:fol-loc}).

\begin{theorem}\label{thm:K}
    Let $p\in C(\surf)$ be a non-degenerate characteristic point (i.e., $\hat K_p\neq-1$), let $\mathscr L_p$ be the set of leaves with $p$ as an endpoint and consider the operator $-\Delta_\ell$ defined in \eqref{lapl div grad}, with domain $\mathcal{C}^\infty_0(\ell)$. 
    \begin{itemize}
        \item If $\hat K_p> -3/4$, then $p$ is a focus and for all leaves $\filo\in \mathscr{L}_p$, the operator $-\Delta_\filo$ is not essentially self-adjoint;
        \item if $\hat K_p= -3/4$, then $p$ is a star node and for all leaves $\filo\in \mathscr{L}_p$, the operator $-\Delta_\filo$ is not essentially self-adjoint;
        \item If $\hat K_p\in(-7/9,-3/4)$, then $p$ is a proper node and for all leaves $\filo\in \mathscr{L}_p$, the operator $-\Delta_\filo$ is not essentially self-adjoint;
        \item If $\hat K_p\in(-1,-7/9]$, then $p$ is a proper node and if $\filo\in \mathscr{L}_p$, the operator $-\Delta_\filo$ is essentially self-adjoint at $p$ unless $\filo$ enters $p$ tangentially to the eigenspace corresponding to the largest eigenvalue;
        \item If $\hat K_p\in[-3,-1)$, then $p$ is a saddle and if $\filo\in \mathscr{L}_p$, the operator $-\Delta_\filo$ is essentially self-adjoint at $p$ unless $\filo$ enters $p$ tangentially to the eigenspace corresponding to the largest eigenvalue;
        \item If $\hat K_p<-3$, then $p$ is a saddle and for all leaves $\filo\in \mathscr{L}_p$, the operator $-\Delta_\filo$ is not essentially self-adjoint.
    \end{itemize}
\end{theorem}

\begin{remark}
    In the above, when speaking about eigenvalues and eigenspaces, we are considering the linearization of the characteristic vector field at $p$. See Remark~\ref{rmk:linearization}.
\end{remark}

We now adopt a global perspective to the self-adjointness problem by considering an operator $\bDelta$ on the whole foliation $\fol$, starting from the leaf-wise operators $\Delta_\ell$. 
Such operator is defined by
\begin{equation}
    \label{direct integrals}
    \bDelta=\int^\oplus_\fol \Delta_\ell \, d\nu(\ell) 
    \qquad\text{acting on}\qquad
    \int^\oplus_\fol L^2(\ell,d\mu_\ell) \, d\nu(\ell),
\end{equation}
where $\nu$ is a measure on the set of leaves of the foliation $\fol$
(see Section~\ref{sec: sa ext}.)
One can then consider the corresponding Schrödinger equation,
\begin{equation}
    \label{schrod}
    i \partial_t \Psi_t = - \bDelta \Psi_t.
\end{equation}

 The operator $\bDelta$ is symmetric on the direct integral $\int_\fol^\oplus \mathcal{C}^\infty_0(\ell)$ and admits self-adjoint extensions. Any such extension either decouples the leaves from one another, or couples some of them. A decoupling extension is constructed by imposing  suitable boundary conditions at every characteristic points {\em leaf by leaf}, namely by defining a specific self-adjoint extension for every $\Delta_\ell$.  Such extensions generate a decoupled dynamics, in the sense that every solution to the Schr\"odinger Equation (10) evolves inside each leaf separately from the rest of the foliation. On the contrary, coupling extensions can be defined by imposing, at some characteristic points, boundary conditions that link different leaves, thus producing coupled dynamics where leaves communicate.

  In the present paper we focus on the decoupled family, while the coupled one will be addressed in future works, together with the existence of an intrinsic measure $\nu$ on the set of leaves of $\fol$. However, we go deeper into this point of view in
  Section \ref{sec: sa ext}.
  

\subsection{Further remarks}

\paragraph{Heat diffusion}
We recall that the unitary transformation that we use to associate a potential $V_\ell$ to $\Delta_\ell$ cannot be used to study the associated heat diffusion. Indeed, while Schrödinger evolutions associated with unitarily equivalent operators remain equivalent too, the same is not true for the heat diffusion where one is interested in properties that are not invariant under these transformations (e.g. Markovianity).

In general, heat diffusion requires different techniques. In \cite{barilariStochastic2021}, the authors showed via probabilistic techniques, that the heat never reaches characteristic points of node and focus type, while it always reaches saddle type characteristic points. However,  Kirchhoff's conditions, admissible only for saddle type characteristic points, have not been studied for the heat equation, up to now.

\paragraph{Other intrinsic laplacians}
    The operator $\Delta_\ell$ studied in this paper is not the only possible intrinsic Laplacian on $\ell\in\fol$. Actually, since every leaf $\ell$ is endowed with a locally Euclidean metric, one could simply study the operator $\partial^2_s$, where $s$ is the arc-length parameter on $\ell$. 
    For the heat equation, this type of operators have been studied by Walsh \cite{walsh} in some special cases. For the Schrödinger equation this has not been studied yet.

    {As widely known (\cite{reedMethods1980}), the restriction $\partial^2_{s,0}$ of $\partial^2_s$ to the space $\mathcal{C}^\infty_0 (\ell)$ 
    is not essentially self-adjoint, unless $\ell$ is isomorphic to $\R$. However, one finds in this case the same dichotomy between the coupling and decoupling family of self-adjoint extensions that we sketched for $\bDelta$.} 
    {Consider for instance a characteristic point $p$ of saddle type. There are four leaves emanating from it, say $\ell_i, \ i = 1,2,3,4$, with corresponding arc-length $s_i$. One can then construct a Laplacian for the whole system by taking the direct sum of some  self-adjoint extensions of the operators $\partial^2_{s_i,0}, \ i = 1,2,3,4$, obtaining then a decoupling Laplacian. 
    Alternatively, one can  interpret $p$ in the language of quantum graphs as a {\em vertex} with four {\em edges} originating from it, so the overall Laplacian  can be singled out by using techniques of quantum graphs \cite{Kostrykin_Schrader_1999,berkolaikokuchment}, i.e. defining a coupling self-adjoint extension, e.g. the Krichhoff's one. }

    On the other hand,  characteristic points of focus and node type  can be understood as vertices emanating an infinite number of edges, that is not typical in quantum graphs, where vertices usually have a finite degree.
    Some cases have been investigated in \cite{mugnolo1}.
    
    Anyway, the present paper is devoted to the investigation of the operator $\Delta_\ell$ instead of $\partial_s^2$, since it appears to be more natural as it can be obtained exploiting a Riemannian approximation scheme (see \cite{barilariStochastic2021}).

\subsection{Structure of the paper}

In Section \ref{sec: volume} we introduce the geometric background, specifically surfaces embedded in 3-dimensional contact sub-Riemannian manifolds. We define the characteristic foliation and the notion of characteristic points. Following this, we introduce the $\laplaciano$ on a leaf, as detailed in Section \ref{sec: Lapl}, and explore its essential self-adjointness. Here we see how the notion of {essential }self-adjointness can be studied by looking locally at the endpoints of a leaf. In particular, we see that for infinite endpoints $\Delta_\ell$ is always locally essentially self-adjoint, while for a characteristic endpoint $p$, the {essential} self-adjointness of $\Delta_\ell$ at $p$ depends on $\widehat{K}_p.$
In Section \ref{sec: sa ext} we analyze the Laplacian as a global operator. In particular, we define an operator on the foliation that, on each leaf, corresponds to $\Delta_\ell$. We then investigate its self-adjoint realizations, focusing on those that yield dynamics confined on the leaves. This case is equivalent to studying self-adjoint extensions on a single leaf, which we approach using Sturm-Liouville theory. An alternative approach that produces the explicit expression of the deficiency spaces is the Von Neumann's theory, that we develop in a particular case.

In Appendix \ref{sec: disi} we collect some results about disintegration of measures on foliations. The purpose here is to show that the measure on the leaf $\ell$, that renders the $\laplaciano$ $\Delta_\ell$ symmetric, comes from the disintegration of the measure $\mu$ on the foliation.

\section{The geometric structure}
\label{sec: volume}
In the following, let $M$ be a smooth connected 3-dimensional manifold with a contact sub-Riemannian structure, i.e., a pair $(D,g)$, where $D$ is a vector distribution and $g$ is an inner product on $D$. In addition it is required that $D$ satisfies the H\"ormander condition, i.e., locally there exist two vector fields  $X_1,X_2$ belonging to $D$
\[
\operatorname{span}\left\{[X_1,X_2],X_1,X_2\right\}\big|_q=T_q M, \qquad \forall q\in M.
\]
The H\"ormander condition is equivalent to the local existence of a \emph{contact form}, i.e., a  one form $\omega\in \Omega^1(M)$ such that
\begin{equation}
\label{contact}
    D=\text{ker }\omega, \ \ \omega\wedge d\omega\ne 0.
\end{equation}
To lighten the exposition, we henceforth make the following assumption
\begin{center}
    \textbf{(H0)} \qquad
    The contact sub-Riemannian manifold is coorientable\\
    (i.e., \eqref{contact} holds globally).
\end{center}
Observe that this implies, in particular, that $M$ is orientable. The word coorientable comes from the fact that the existence of a global contact form is equivalent to the orientability, and hence triviality, of $TM/D$ (for more details see \cite{geiges-contact}).
We will discuss in Remark~\ref{rmk:non-coorientable} the non-coorientable case.

Note that if $\omega$ is a one-form satisfying \eqref{contact}, then so is also $f\omega$ for any smooth never-vanishing function $f$. Thanks to the presence of the metric $g$, we normalize $\omega$ canonically by imposing that $d\omega$ restricted to $D$ is the Euclidean volume given by $g$ on the distribution.

Given a 3-dimensional contact sub-Riemannian manifold, one can use the contact form $\omega$ to define a canonical vector field $X_0$ that is transversal to the vector distribution in each point. This is called the \textit{Reeb vector field} and it is defined formally as the only vector field $X_0$ for which
\[
\omega(X_0)\equiv 1 
\quad\text{and}\quad
d\omega(X_0,\cdot)=0.
\]
Using this vector field, one can define a volume on $M$, called the Popp volume, as the unique 3-form $\mathcal{P}$ such that $\mathcal{P}(X_1,X_2,X_0)=1$, where $X_1,X_2$ is an orthonormal frame of $(D,g)$ and $X_0$ is the Reeb vector field. 
In particular, thanks to the above normalization, the Popp volume coincides with $d\omega\wedge \omega$.

Let now $S$ be a smooth surface embedded in $M$ (actually, $\mathcal{C}^3$ regularity would be sufficient). 
Recall that $C(S)$ is the set of characteristic points, i.e., the set of points $p\in S$ such that $D_p= T_p S$. All along the paper we assume that
\begin{center}
    The set $C(S)$ is made of isolated points.
\end{center}
We also make the following assumption
\begin{center}
\textbf{(H1)}\qquad
    $S$ is orientable.
\end{center}

Then, the surface volume form defined in \eqref{hor unit norm} using the Popp volume $\mathcal{P}$ is global on $S\setminus C(S)$. Furthermore, it can be extended to a measure on $S$ by continuity, see Proposition~\ref{prop: mu va a zero}.

\subsection{The characteristic foliation}
\label{sec:characteristic-foliation}
In Section \ref{sec: intro} we defined the map $\dir :\surf\setminus C(S) \ni q\mapsto D_q\cap T_q\surf$, which associated to each non-characteristic point $q$ a direction, called \emph{characteristic direction}, corresponding to the intersection $D_q\cap T_qS$. 
The integral curves of $\dir$ define the \emph{characteristic foliation} $\fol$ of $S\setminus C(S)$, made of $1$-dimensional leaves, each denoted by $\filo$.

Recall that each leaf $\filo$ is endowed with a locally Euclidean metric, inherited from the Riemannian metric on $D$. Hence, $\ell$ is isometric to one of the following:
\begin{itemize}
    \item a circle;
    \item the real line $\mathbb{R}$;
    \item the half-line $(0,+\infty)$;
    \item a finite segment $(a_0,a_1)$ with $-\infty<a_0<a_1<\infty$.
\end{itemize}

We would now like to define, locally, a class of vector fields which span such characteristic direction.

\begin{definition}
\label{def cvf}
    Let $S$ be a smooth surface embedded in a 3D contact sub-Riemannian manifold $M$ and let $U$ be an open subset of $S$.
    We say that a smooth vector field $X$, defined on $U$, is a \emph{characteristic vector  field} for $S$ on $U$ if 
    \[
    \operatorname{span}X(q)=
    \begin{cases}
        \{0\} & \text{if }q\in C(S) \\
        D_q\cap T_qS& \text{otherwise,}
    \end{cases}
    \]
and it holds div$X(p)\ne 0$ for each $p\in C(S)$.
\end{definition}

Locally, a characteristic vector field always exists. Indeed, if $S$ is locally described as the zero locus of a smooth function $u$, with $du\neq 0$, given  an orthonormal frame $\{X_1,X_2\}$ for $D$, the vector field
\begin{equation}
    \label{eq:local-char-vf}    
    X = (X_2u)X_1 - (X_1u)X_2,
\end{equation}
is characteristic. Observe, in particular, that $\operatorname{div} X(p) = [X_1,X_2]u(p) = X_0 u(p)$ and hence is non vanishing if $p\in C(S)$.

Thanks to \textbf{(H0)} and \textbf{(H1)}, there exists a global characteristic vector field (see, e.g., \cite{geiges-contact}).
We additionally require the following:
\begin{center}
    \textbf{(H2)}\qquad
    The characteristic vector field is complete.
\end{center}
This choice is meant to simplify the exposition. Actually, under this assumption any finite endpoint of a leaf is a characteristic point (see Proposition 1.3 of  \cite{barilariInduced2022}).

Normalizing a global characteristic vector field, one obtains a unitary vector field 
\[\XS=\frac{X}{\norm{X}} \qquad \text{on } S\setminus C(S).\] 
Such vector field is called the \textit{unitary characteristic vector field}, and is independent of the starting characteristic vector field $X$, up to a sign. 
From \eqref{eq:local-char-vf}, locally one has

\begin{equation}
    \label{XS}
    \XS=\pm\frac{(X_2u)X_1-(X_1u)X_2}{\sqrt{(X_1u)^2+(X_2u)^2}}.
\end{equation}

\begin{remark}
    \label{rmk:non-coorientable}
    When at least one between \textbf{(H0)} and \textbf{(H1)} is not satisfied, the existence of a global characteristic vector field is not guaranteed.  
    Nevertheless, the map $\dir$ and the characteristic foliation $\fol$ are still well defined. In particular, each leaf $\ell$ is still endowed with a Riemannian structure.
    Moreover, the surface measure $\mu$ is not in general a 2-form, but only a density.
\end{remark}

\subsection{Local structure near characteristic points}
\label{sec:curvature}
We now introduce the curvature-like invariant $\widehat{K}_p$ (first defined in \cite{barilariInduced2022}), which is a real number that is associated to each characteristic point and classifies the qualitative behaviour of the characteristic foliation near those point. 

Since this section is concerned with local considerations, assumptions \textbf{(H0)}, \textbf{(H1)}, and \textbf{(H2)} are not necessary.

\begin{definition}
Let $p\in C(S)$ be a characteristic point of $S$. Suppose that $S$ is described, in a neighbourhood of $p$, as the zero locus of a smooth function $u$, with $du\ne 0$. Then, the curvature of $p$ is 
\[
\widehat{K}_p=-1+\frac{\text{det Hess}\,u (p)}{[X_1,X_2]\,u(p)^2},
\]
where $\{X_1, \, X_2\}$ is an orthonormal frame of $D$ with respect to $g$.
\end{definition}

\begin{remark}
    This quantity is intrinsic in the sense that it does not depend neither on the choice of the function $u$ nor on the choice of the orthonormal frame $X_1,X_2$. 
    Moreover, it can be alternatively defined via a Riemannian approximation scheme, as the limit of the ratio between the Gaussian curvature and a normalizing factor, which is the determinant of a bilinear form defined on the vector distribution, see \cite{barilariInduced2022}.
\end{remark}

The following characterization of $\widehat{K}_p$ is proved in \cite[Corollary~4.3]{barilariInduced2022}.

\begin{proposition}
    \label{prop:cannarsa-theorem}
    Let $\lambda_1, \lambda_2$ be the eigenvalues of $\Hess u(p)\circ J$, where $J$ is the operator defined by $g(X,J(Y)) = d\omega(X,Y)$ for all $X,Y\in D$ and $u$ is a smooth function that defines $\surf$ locally around $p$. 
    Then, up to multiplying $u$ by a constant, one can assume that $\lambda_1+\lambda_2=[X_1,X_2]u(p)=1$, and one has
    \begin{equation}
        \lambda_{1,2}=\frac{1}{2} \pm \sqrt{-\frac{3}{4}-\widehat{K}_p}.
    \end{equation}
\end{proposition}

\begin{remark}
    \label{rmk:linearization}
    Since $\Hess u(p)\circ J = DX(p)$, where $X$ is the characteristic vector field defined in \eqref{eq:local-char-vf}, when $\lambda_{1,2}$ are real, the corresponding eigenspaces determine the directions in which the foliation approaches the characteristic point.
    Moreover, the normalization $\lambda_1+\lambda_2=1$ corresponds to the choice \[X = \frac{(X_2 u) X_1 - (X_1 u)X_2}{X_0u},\] which is intrinsic in a neighbourhood of $p$ (see \cite[Remark~4.4]{barilariInduced2022}). 
\end{remark}

Depending on the value of the above curvature one can deduce the local qualitative behaviour of the characteristic foliation around a characteristic point (see Figure~\ref{fig:fol-loc}).
In particular, we have the following result (see \cite[Corollary~4.5]{barilariInduced2022}).
Observe that characteristic points of circle type are excluded, as every 3-dimensional contact sub-Riemannian manifold is locally tight (see \cite{geiges-contact}).

\begin{figure}
    \centering
     \begin{minipage}
        {.24\textwidth}
        \includegraphics[width=\textwidth]{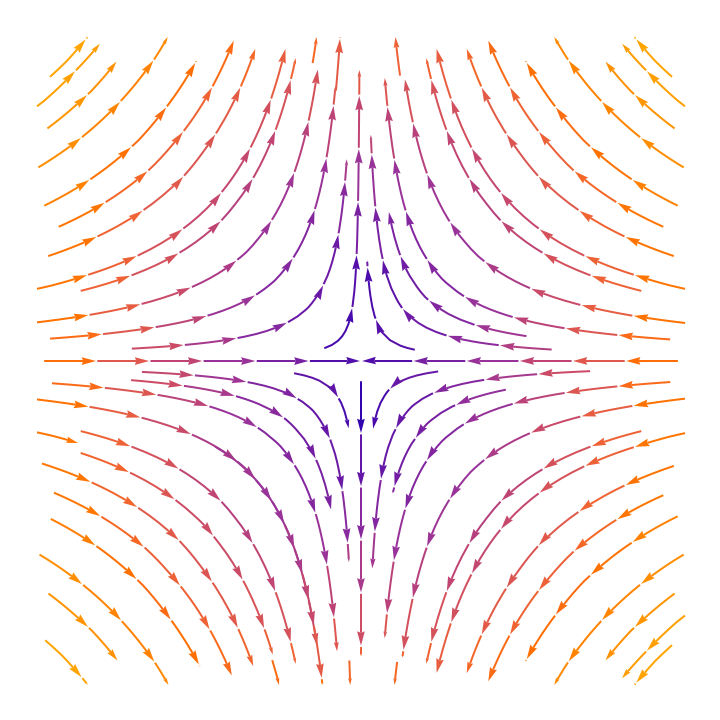}\\
        \centering 
        $\widehat{K}_p<-1$\\
        Saddle
    \end{minipage}
     \begin{minipage}
        {.24\textwidth}
        \includegraphics[width=\textwidth]{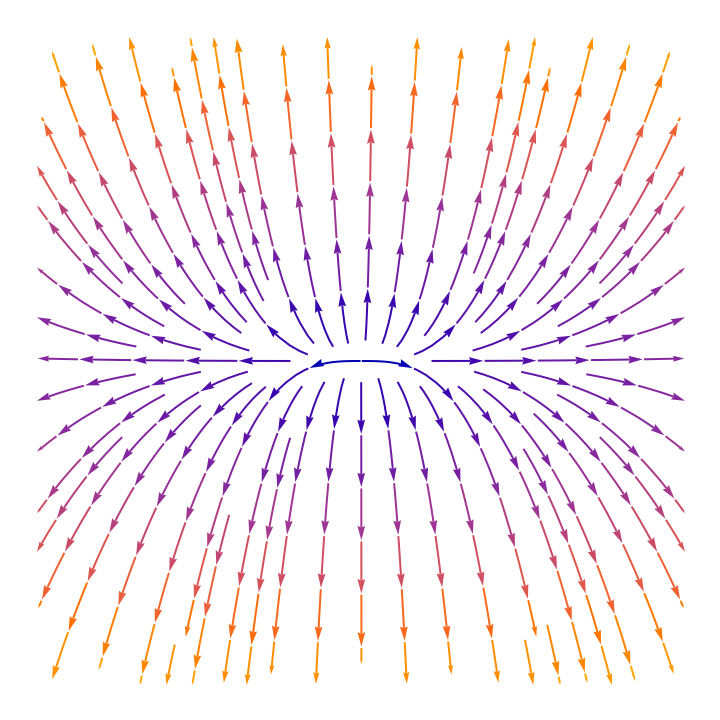}\\
        \centering 
        $\widehat{K}_p\in(-1,-\frac{3}{4})$\\
        Proper node
    \end{minipage}
     \begin{minipage}
        {.24\textwidth}
        \includegraphics[width=\textwidth]{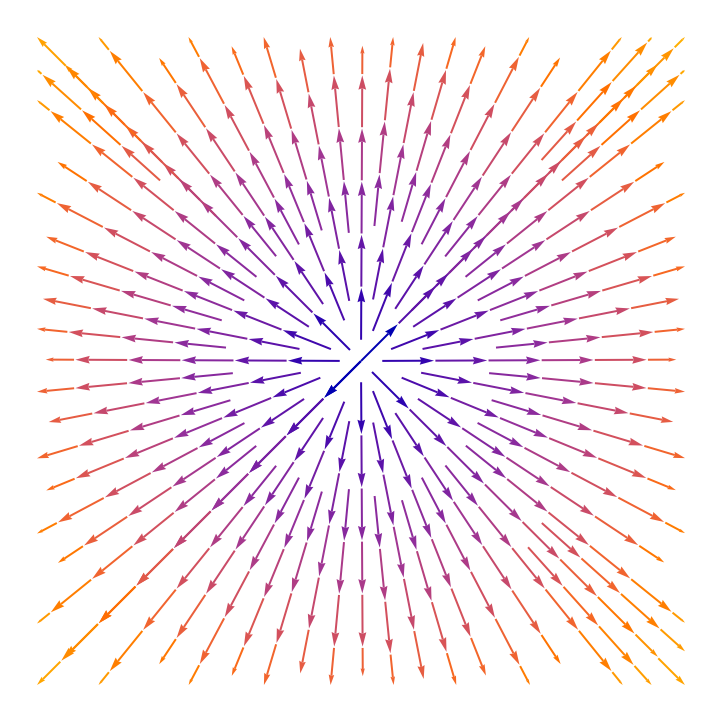}\\
        \centering 
        $\widehat{K}_p=-\frac{3}{4}$\\
        Star node
    \end{minipage}
     \begin{minipage}
        {.24\textwidth}
        \includegraphics[width=\textwidth]{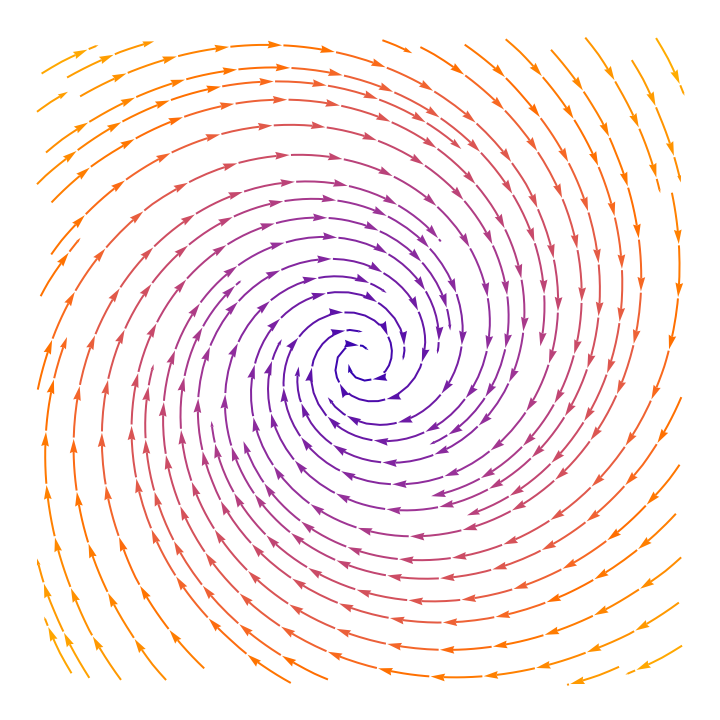}\\
        \centering 
        $\widehat{K}_p>-\frac{3}{4}$\\
        Focus
    \end{minipage}
    \caption{Qualitative picture for the characteristic foliation at an isolated characteristic point, with the corresponding values for $\widehat{K}_p$. From left to right, we recognise a saddle, a proper node, a star node and a focus.
    }
    \label{fig:fol-loc}
\end{figure}

\begin{proposition}
    \label{prop:classification-foliation}
    Let $S$ be a surface embedded in a 3D contact sub-Riemannian manifold and let $p$ be a non-degenerate characteristic point of $S$, i.e., $\widehat{K}_p\neq -1$. Then, in a neighborhood of $p$ the characteristic foliation $\fol$ is $C^1$-conjugate to
    \begin{itemize}
        \item a saddle if $\widehat{K}_p\in (-\infty,-1)$;
        \item a proper node if $\widehat{K}_p\in (-1,-3/4)$; 
        \item a star node if $\widehat{K}_p=-3/4$;
        \item a focus if $\widehat{K}_p\in (-3/4,\infty)$.
    \end{itemize}
\end{proposition}

The case when $\widehat{K}_p$ is equal to -1 is degenerate and the behaviour of the foliation is not uniquely defined by $\widehat{K}_p$, for example it can be half saddle-like and half node-like. In this case the Hessian of $u$ is degenerate. It is well known that generically characteristic points are non-degenerate. We will exclude this case in this paper (see \cite{barilariInduced2022}).

\section{The Laplacian on leaves}
\label{sec: Lapl}

Recall that the Laplace-Beltrami-like operator $\Delta_\ell$ on a given leaf $\ell$ of $\fol $ is defined in \eqref{lapl div grad} as $\Delta_\ell=\div_\mu \circ \nabla_\ell$. Here, $\nabla_\ell$ is the Riemannian gradient associated with the locally Euclidean structure on $\ell$, and $\div_\mu$ is the divergence w.r.t.~the measure $\mu$ induced on $S\setminus C(S)$ by the Popp volume on $M$.
The unitary characteristic vector field $\XS$ defined in \eqref{XS} is actually an orthonormal frame for the Euclidean structure on each leaf $\ell\in \fol$. Hence,
\begin{equation}
\label{lapl char}
 \Delta_\filo = \text{div}_\mu\nabla_\ell = \XS^2 + \div_\mu(\XS) \XS.
\end{equation}
Notice that the above expression is well defined even if $\textbf{(H0)}$ and $\textbf{(H1)}$ are not satisfied, since it does not depend on the sign of $\XS$.

The purpose of this section is to study the {essential} self-adjointness of $\Delta_\ell$ on the different kind of 1-dimensional leaves introduced in Section~\ref{sec:characteristic-foliation}: circle, line, half-line and segment.  

\subsection{Local expression of $\Delta_\ell$}
\label{local b}

Fix the arc-length coordinate $s$ on $\ell$. 
Following the notation of \cite{barilariStochastic2021}  we denote by $b(s)$ the quantity $\operatorname{div}_\mu(\XS)$ computed w.r.t.~$s$. Then, we have
    \begin{equation}
    \label{lapl-s}
        \Delta_\filo = \partial_s^2 + b(s)\partial_s.
    \end{equation}
This is a symmetric operator on $L^2(\filo,\nu(s)\,ds)$, where $\nu$ is any measure on $\filo$, which satisfies $b = \partial_s(\log\nu)$. 
This measure $\nu$ can actually be obtained by disintegrating the surface measure $\mu$ along the leaves of $\fol$. For more details on this, see Section~\ref{sec: disi}.

We map the first order term in \eqref{lapl-s} to a potential via the unitary transformation $T : L^2(\filo,\mu_\ell\,ds) \to L^2(\filo)$ defined by $Tu = \sqrt{\mu_\ell}\,u$. 

\begin{lemma}
    \label{lem:potential}
    The operator $-\Delta_\ell$ is unitarily equivalent to the following operator on $L^2(\filo,ds)$:
    \begin{equation}
    \label{op:potential}
        L_\filo = -\partial_s^2 + V(s), \qquad V = \frac{1}{2}\partial_s b+ \frac14 b^2.
    \end{equation}
\end{lemma}

In the following, we characterize the expression of $b(s)$, and hence $V(s)$, close to characteristic points.
\begin{proposition}
    \label{prop:b-asymp}
    Assume that $\ell\sim (0,a_1)$ and $0$ corresponds to the characteristic point. Then, there exist two functions $\eta,\zeta\in \mathcal{C}^\infty([0,a_1))$ smooth up to $s=0$, such that for $s\in (0,a_1)$, one has
    \begin{itemize}
        \item if $p$ is a focus or a star node, then
        \begin{equation}
            b(s) = \frac{2}{s} + \eta(s)
            \qquad\text{and}\qquad
            V(s) = \frac{\zeta(s)}s,
        \end{equation}
        \item if $p$ is a proper node or a saddle and $\filo$ is tangent to the $\lambda$-eigenspace of $\Hess\phi(p)\circ J$ at $p$, then
        \begin{equation}
            b(s) = \frac{1}{\lambda s} + \eta(s)
            \qquad\text{and}\qquad
            V(s) = \frac{1-2\lambda}{4\lambda^2}\frac1{s^2} + \frac{\zeta(s)}s.
        \end{equation}
    \end{itemize}
\end{proposition}

\begin{proof}
    Let $\eta(s) := b(s) - 2/s$ in the focus or star node case, and $\eta(s)=b(s)-1/(\lambda s)$ in the proper node or saddle case.     
    In \cite[Lemma~3.1 and Lemma~3.2]{barilariStochastic2021} it is shown that $\eta(s)=O(1)$ and that it is defined up to $s=0$. From the fact that $b(s)$ is smooth on $(0,a_1)$,  following the computations in \cite{barilariStochastic2021} one obtains that $\eta\in \mathcal{C}^\infty([0,a_1))$, as claimed.

    The expression of $V$ is readily obtained by the expression in Lemma~\ref{lem:potential}, which yields
    \begin{equation}\label{eq:zeta}
    \zeta(s) = \frac{\eta(s)}{2\lambda} + \frac{s}{4}\left(2\eta'(s)+\eta(s)^2\right).     
    \end{equation}
    In particular, $\zeta$ is smooth up to $s=0$.
\end{proof}

\begin{remark}
    By \eqref{eq:zeta} it follows that $\eta(0)=0$ if and only if $\zeta(0)=0$, and that $\eta\equiv0$ if and only if $\zeta\equiv 0$. In particular, this implies that whenever the remainder $\eta$ is zero in the expansion of $b$, the potential is a pure inverse square potential.
\end{remark}

\subsection{Essential self-adjointess of $\Delta_\ell$}
\label{sec: ess-self}

In this section we present the proof of Theorem~\ref{thm:K}.
We start by recalling the following well-known result relating the completeness of a Riemannian manifold with the self-adjointness of the divergence of the gradient independently on the choice of the volume (see for instance \cite{STRICHARTZ198348,Braverman_2002}).

\begin{theorem}
\label{riemannian}
    Let $(M,g)$ be a Riemannian manifold and let $\mu$ be a positive and smooth measure on $M$. If the metric $g$ is complete, then the operator $\Delta_\mu=\div_\mu\nabla_g$ is self-adjoint.
\end{theorem}

This result applies directly to the case in which $\ell$ is either isometric to a circle or to $\R$, yielding the following.

\begin{proposition}
    If $\filo$ is isometric to a circle or a line, then the operator $-\Delta_\filo$ with domain $\mathcal{C}^\infty_0(\filo)$ is essentially self-adjoint on $L^2(\ell,\mu\,ds)$. 
\end{proposition}

\begin{remark}
   The operator $\Delta_\ell$ is essentially self-adjoint even on an infinite leaf with non-empty limit set. For example, this can occur when the limit set is a homoclinic trajectory or a heteroclinic cycle, joining saddle points (see Figure~\ref{fig: eteroclina}). We stress that this is independent of whether the operator $\Delta_\ell$ is {essentially} self-adjoint on the leaves composing the limit set. 
\end{remark}
\begin{figure}
    \centering
    \includegraphics[width=0.37\linewidth]{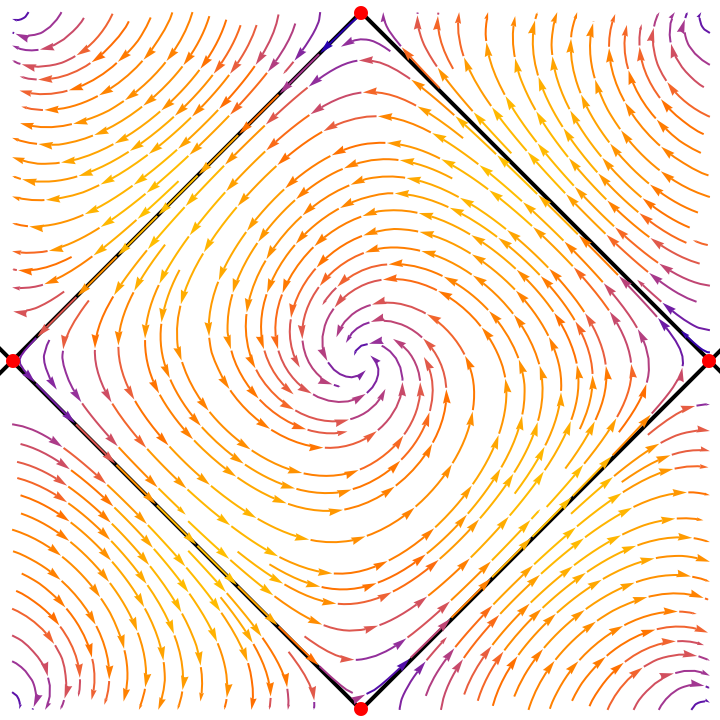}
    \caption{Heteroclinic cycle between 4 saddles. Even if the operator is not essentially self-adjoint on the black leaves composing the cycle, it is nevertheless essentially self-adjoint at infinity on every leaf in the interior of the cycle.}
    \label{fig: eteroclina}
\end{figure}

We now consider the case in which $\ell\sim (0,a_1)$, $0<a_1\le +\infty$. By assumption \textbf{(H2)}, the endpoints $0$ and $a_1$, when finite, correspond to characteristic points. For simplicity we will discuss the essential self-adjointness of the operator $L_\ell=-\partial_s^2+V(s)$ defined in $\mathcal{C}_0^\infty (\ell)$, which by Lemma~\ref{lem:potential} is equivalent to that of $\Delta_\ell$.

The theory of self-adjointness of second order Schr\"odinger operators on an interval was developed by Weyl (see \cite{reedFourier2007}). Actually the self-adjointness depends on the asymptotic behaviour of $V$ at $0$ and $a_1$.

\begin{definition}
    We say that the operator $L_\ell$ is in the limit circle case at $a_1$ (respectively at zero) if all solutions of $L_\ell \phi=0$ are square integrable at $a_1$ (respectively at zero). Otherwise, we say that $L_\ell$ is in the limit point case.
\end{definition}

The basic result of this theory is the following (\cite[Theorem X.7]{reedFourier2007}).
\begin{proposition}
   The operator $L_\ell$ is essentially self-adjoint if and only if it is in the limit point case at both endpoints of $\ell$. 
\end{proposition}

Inspired by the above result, we give the following definition.
\begin{nomenclature}
     With a slight abuse of language, 
     we say that the operator $L_\ell$ is essentially self-adjoint at an endpoint of $\ell$ if the operator is in the limit point case at such endpoint.
\end{nomenclature}

Using again Theorem \ref{riemannian} and the unitarily equivalence of $L_\ell$ and $\Delta_\ell$, one readily obtains the following. 

\begin{proposition}
\label{prop:sa-infinity}
 If $a_1=\infty$, then the operator $L_\ell$ is essentially self-adjoint at $a_1$.  
\end{proposition}

It remains to study the essential self-adjointness at $0$. The case for finite $a_1$ can be traced back to this one.

We characterize the endpoint-wise essential self-adjointess of $L_\filo$ in the following.

\begin{lemma}
\label{lem:ess self adj}
    Let the non-degenerate characteristic point $p$ be an endpoint of the leaf $\filo$.
    Then, the operator $L_\filo$ is essentially self-adjoint at $p$ if and only if all eigenvalues of $\Hess u(p)\circ J$ are real and $\filo$ enters $p$ tangentially to the $\lambda$-eigenspace, where $\lambda\in [-1,1/3]$.
\end{lemma}

\begin{proof}
    Recalling the expression of $L_\ell$, introduced in Lemma \ref{lem:potential} and the estimates of Proposition \ref{prop:b-asymp}, we get:
    \[
    L_\ell=-\partial^2_s+\frac{c}{s^2}+\frac{\zeta(s)}{s}.
    \]
    Here, $c=0$ in the focus or star node case, while $c=(1-2\lambda)/({4\lambda^2})$ in the saddle or proper node case.
    
    Assume now that $c\ge{3}/{4}$, i.e., $\lambda\in[-1,1/3]$. We know that in this case the operator
    \[
    A=-\partial^2_s+\frac{c}{s^2}
    \]
    is essentially self-adjoint (see \cite[Theorem~X.10]{reedFourier2007}). 
    We apply KLMN Theorem (see \cite[Theorem~X.17]{reedFourier2007}) to show that so is $L_\ell$. To this purpose, let $\beta$ be the following quadratic form
    \[
    \beta(u,u)=\int_0^1\frac{\zeta(s)}{s}\abs{u}^2\,ds.
    \]
    We need to find some $a>1$ and $b\in\R$ such that the following estimate for $\beta$ holds:
    \begin{equation}
    \label{klmn}
    \begin{split}
           \abs{\beta(u,u)}&\le a(u,Au)+b(u,u)\\
           &=a\left[\int_0^1-u\,\partial^2_suds+\int_0^1\frac{c}{s^2}u^2\,ds\right]+b\int_0^1u^2ds.
    \end{split}
    \end{equation}
    To prove this we split the integral of $\beta$ into two parts:
    \[
    \abs{\int_0^1\frac{\zeta(s)}{s}\abs{u}^2ds}\le\abs{\int_0^\epsilon\frac{\zeta(s)}{s}\abs{u}^2ds}+\abs{\int_\epsilon^1\frac{\zeta(s)}{s}\abs{u}^2ds}.
    \]
    To bound the integral computed close to zero we use the fact that, being $\zeta(s)$ continuous, for any $a>1$ we can find $\epsilon>0$ small enough such that
    \[
    \abs{\frac{\zeta(s)}{s}}\le a\frac{c}{s^2}, \qquad \text{for } s\in (0,\epsilon].
    \]
    In particular,
    \[
   \int_0^\epsilon \abs{\frac{\zeta(s)}{s}}\abs{u}^2\,ds\le a\int_0^1\frac{c}{s^2}\abs{u}^2\,ds.
    \]
    Regarding the second part of the integral of $\beta$, using that $\frac{\zeta(s)}{s}$ is continuous in $[0,1]$, integration by part and Poincarè inequality, we get a constant $k$ such that
    \[
    \abs{\int_\epsilon^1\frac{\zeta(s)}{s}\abs{u}^2ds}\le -k\int_0^1u\partial_s^2u\,ds.
    \]

To sum up, for any $a>1$ there exists $k\ge 0$, such that
\[
\abs{\beta(u,u)}\le a\int_0^1\frac{c}{s^2}\abs{u}^2ds-k\int_0^1 u\partial_s^2u ds.
\]
This implies,
\[
\abs{\beta(u,u)}\le \max\{a,k\}\left[\int_0^1\frac{c}{s^2}\abs{u}^2ds-\int_0^1 u\partial_s^2u ds\right].
\]
Hence \eqref{klmn} holds. In particular, we can apply KLMN Theorem and conclude that $L_\ell$ is essentially self-adjoint.

    To conclude the proof of the statement, we now assume that $c<\frac{3}{4}$. In this case $A$ is not essentially self-adjoint. 
    To check the essential self-adjointness of $L_\ell$ in $0$, we check it on an interval like $(0,\delta)$ for some $\delta$.
    We can choose $\delta$ in such a way that, on $(0,\delta)$ the following inequality holds:
    \[
    \frac{c}{s^2}+\frac{\zeta(s)}{s}<\frac{3}{4}\frac{1}{s^2}.
    \]
    Applying \cite[Theorem~X.10]{reedFourier2007}, we conclude that if $A$ is not essentially self-adjoint, then not even $L_\ell$ is {essentially} self-adjoint. 
\end{proof}

Combining Lemma~\ref{lem:ess self adj} with Proposition~\ref{prop:cannarsa-theorem}, and recalling that the essential self-adjointness of $\Delta_\ell$ and $L_\ell$ at a point is equivalent, we complete the proof of Theorem~\ref{thm:K}.

\section{The foliated Laplacian on $S$}
\label{sec: sa ext}

In this last section we deal with the problem of the self-adjoint extensions of the operator $\bDelta$ introduced  in Section~\ref{sec:intro-lapl-fol}.

Let us discuss the choice of the operator $\bDelta$ given in \eqref{direct integrals}. A key aspect of defining an operator on the whole foliation consists in singling out the appropriate Hilbert space on which it should act. A first attempt could be to define $\bDelta$ on the direct sum of the $L^2(\ell,d\mu_\ell)$ spaces associated with each leaf in the foliation. However, even though this method does produce a well-defined operator acting on functions that are non-zero on at most countably many leaves, it lacks a clear geometric interpretation.
This justifies the introduction of a measure $\nu$ on the leaves of $\fol$ and the choice of $\int_\fol^\oplus L^2(\ell,d\mu_\ell)\,d\nu(\ell)$ as an Hilbert space in \eqref{direct integrals} (see \cite{vonNeumann1939}).

As already pointed out, among the possible self-adjoint extensions extensions of $\bDelta$ with domain $\int_\fol^\oplus \mathcal{C}^\infty_0(\ell)\,d\nu(\ell)$ one can distinguish between those that yield a mixed dynamics, where interaction between leaves occurs, and those that result in a disjoint dynamics, where each leaf evolves independently. 
Here, we limit ourselves to self-adjoint extensions of the operator $\bDelta$ that give rise to disjoint dynamics, i.e., dynamics in which there is no communication between different leaves.
These extensions are constructed by taking the direct integral of self-adjoint extensions on individual leaves. This approach naturally ensures that the resulting global dynamics is a collection of independent dynamics, each confined to a single leaf. 
Hence, we start by classifying self-adjoint extensions of $\Delta_\ell$ on a single leaf $\ell$.

We need the following.

\begin{definition}
    \label{def:lambda0}
    Let $p\in C(S)$ be an endpoint of a leaf $\ell\in\fol$, and let $\lambda_{1,2}$ be the eigenvalues of $\operatorname{Hess}u(p)\circ J$ normalized such that $\lambda_1+\lambda_2=1$.
    Then, define $\lambda_0(\ell,p)\in\mathbb{R}$ as follows
    \begin{itemize}
        \item If $p$ is a proper node or a saddle, then $\lambda_0(\ell,p) = \lambda_i$ where $i\in\{1,2\}$ is  such that $\ell$ enters $p$ tangentially to the $\lambda_i$-eigenspace.
        \item If $p$ is a focus or a star node, then $\lambda_0(\ell,p)=\operatorname{Re}\lambda_1=\operatorname{Re}\lambda_2=1/2$.
    \end{itemize}
    In the following, when no ambiguity arises, we refer simply to $\lambda_0$, dropping the dependence on $\ell$ and $p$.
\end{definition}

Let $\ell$ be a leaf isometric to $(0,a_1)$, with $a_1\le \infty$, and let $s$ be the arc-length parameter\footnote{
We already saw that, on unbounded leaves, the operators $L_\ell$ are essentially self-adjoint at infinity. For this reason we only need to consider self-adjoint extensions on leaves that contain at least one characteristic point in the limit circle case.
}. 
To classify the self-adjoint extensions of $\Delta_\ell$ on $L^2(\ell,d\mu_\ell)$, it is more convenient to look at the unitarily equivalent operator $L_\ell$ on $L^2(\ell,ds)$.
By Proposition \ref{prop:b-asymp}, if $a_1=\infty$, then the operator $L_\ell$ has the following form:
\[
L_\ell=-\partial_s^2+\frac{\beta}{s^2}+\frac{\zeta(s)}{s}, 
\qquad\text{where }
\beta = \frac{1-2\lambda_0}{4\lambda_0^2}.
\]
Similarly, if $a_1<\infty$, then there exist two functions $\zeta_0$ and $\zeta_1$, smooth on $[0,a_1]$, and two constants $\beta_0,\beta_1\in \R$ such that
\begin{equation}
    \label{eq: 0,1}
L_\ell=-\partial_s^2+\frac{\beta_0}{s^2}+\frac{\zeta_0(s)}{s}+\frac{\beta_1}{(s-a_1)^2}+\frac{\zeta_1(s)}{s-a_1}.
\end{equation}
In the following we focus on the case where $L_\ell$ is not essentially self-adjoint in $0$.

In order to study self-adjoint extensions of $L_\ell$, we use Sturm-Liouville's theory for self-adjoint extensions (see, e.g., \cite{zettlSturmLiouville2012}).
According to this theory, one characterizes the domains of self-adjointness of the operator $L_\ell$ as appropriate restrictions of its maximal domain, i.e., the domain of the adjoint of $L_\ell$, characterized as 
\[
D_{\max}(L_\ell)=\left\{u\in L^2(\ell,ds): L_\ell u\in L^2(\ell,ds)\, \right\}.
\]
We also let $D_{\min}(L_\ell)\subset D_{\max}(L_\ell)$ to be the domain of the closure of $L_\ell$. 

We start by characterizing the maximal domain $D_{\max}(L_\ell)$.

\begin{lemma}
\label{dmaxx ns}
    Let $\ell\sim(0,a_1)$ be a leaf of the foliation $\fol$ whose finite endpoints correspond to a characteristic point of proper node or saddle type. Assume that $\ell$ is not essentially self-adjoint at $0$ and fix $\delta>0$. 
    Then, 
    \begin{itemize}
        \item 
    if $L_\ell$ is essentially self-adjoint at $a_1$, then we have
    \[
 D_{\max}(L_\ell)=D_{\min}(L_\ell)+\operatorname{span}\{\phi_1,\phi_2\}
    \]
    where $\phi_1,\phi_2$ are supported in $(0,\delta)$ and, letting $\mu=\sqrt{\frac{1}{4}+\beta}$ and $\gamma=\zeta_0(0)$, we have for $s\in (0,\frac{\delta}{2})$
    \begin{equation}
        \label{phi zero}\phi_1(s)=s^{\frac{1}{2}-\mu}+\frac{\gamma}{1-2\mu}s^{\frac{3}{2}-\mu},
 \quad \text{and}\quad
 \phi_2(s)=s^{\frac{1}{2}+\mu}.
    \end{equation} 

 \item  if $L_\ell$ is not essentially self-adjoint at both $0$ and $a_1$, then we have that $a_1<\infty$. In this case, if $\delta<a_1$ we have
 \[
D_{\max}(L_\ell)=D_{\min}(L_\ell)+\operatorname{span}\{\phi^0_1,\phi^0_2,\phi^1_1,\phi^1_2\}.
    \]
Here, $\phi_i^0(s)=\phi_i(s)$ with $\beta=\beta_0$ and $\gamma=\zeta_0(0)$, and $\phi_i^1(s)=\phi_i(a_1-s)$ with $\beta=\beta_1$ and $\gamma=\zeta_1(1)$. 
     \end{itemize}   
\end{lemma}
\begin{proof}
Observe that if $L_\ell$ is not essentially self-adjoint at $a_1$,{ then from Proposition~\ref{prop:sa-infinity}  it holds $a_1<+\infty$. }
To explicitly describe the maximal domain $D_{\max}(L_\ell)$, one must identify two functions supported near each limit circle endpoint, belonging to the maximal domain, and independent, modulo functions in $D_{\min}(L_\ell)$.
We provide an argument for the endpoint $0$ in the case where $a_1<+\infty$, the other cases being similar. 

We introduce the operator $P_{\beta,\gamma}$ on $(0,\infty)$
\[
P_{\beta,\gamma}=-\partial_s^2+\frac{\beta}{s^2}+\frac{\gamma}{s}.
\]
Let $u\in D_{\max}( P_{\beta,\gamma})$.
Consider a smooth cut-off function $\chi:[0,\infty)\to [0,1]$ which is $1$ on $[0,\frac{\delta}{2}]$ and zero on $[\delta,+\infty)$.
We start by claiming that if $\beta=\beta_0$ and $\gamma=\zeta_0(0)$, then the function $\chi\,u$, possibly restricted to $(0,a_1)$, is in the maximal domain of $L_\ell$. Analogously, if $\beta=\beta_1$ and $\gamma=\zeta_1(1)$, then the function $s\mapsto \chi(a_1-s) u(a_1-s)$ is in $D_{\max}(L_\ell)$.

Direct computations yield
\begin{equation}
\label{max uguale}
    \begin{split}
        L_\ell \,(\chi u)
        &= -\chi(L_\ell \, u)+2\chi'u'-\chi'' \,u\\
        &= \chi P_{\beta,\gamma} u+ \left(\frac{\zeta_0(s)-\gamma}{s}+\frac{\beta_1}{(a_1-s)^2}+\frac{\zeta_1(s)}{a_1-s}\right)\chi u+2\chi'u'-\chi'' \,u. 
    \end{split}
\end{equation}
Observe that $u\in H^1_{\text{loc}}(0,+\infty)$ and that $\abs{\zeta_0(s)-\gamma}\le c s$ for some $c\ge0$. Hence, by definition of $u$ and of the cut-off function $\chi$, we get that $\chi \,u \in D_{max}(L_\ell)$. 
The second part of the claim follows similarly.

To complete the proof of the statement, notice that $\mu\in(0,1)$ immediately implies that the functions defined in \eqref{phi zero} are in $L^2(\ell,ds)$. It remains to show that $P_{\beta,\gamma}(\phi_i)\in L^2(0,\delta)$. This follows directly by computing
\[P_{\beta,\gamma}(s^{\frac{1}{2}\pm \mu})=\gamma s^{-\frac{1}{2}\pm \mu},\quad\text{and}\quad P_{\beta,\gamma}(s^{\frac{3}{2}- \mu})=(2\mu-1)
s^{-\frac{1}{2}- \mu} +\gamma s^{\frac{1}{2}- \mu},
\]
so that in particular the functions defined in \eqref{phi zero} are in the maximal domain of $P_{\beta,\gamma}$ with the correct choices of $\beta$ and $\gamma$.
\end{proof}

It remains the case in which $0$ (respectively $a_1$) is a characteristic point of focus type or star-node type. We now report the equivalent lemma in the case of an half-line leaf $\ell\sim (0,+\infty)$ with a characteristic point of focus type at zero. The finite case can be done as before.

\begin{lemma}
\label{dmax focus}
    Let $\ell\sim (0,+\infty)$ be a leaf with 0 corresponding to a characteristic point of focus or star-node type. Let $\delta>0$. Then
    \[
D_{\max}=D_{\min}+\operatorname{span}\{\phi_1,\phi_2\},
    \]
    where $\phi_1,\phi_2$ are supported in $(0,\delta)$ and, letting $\gamma=\zeta(0)$, for $s\in \left(0,\frac\delta2\right)$ it holds 
\[
\phi_1(s)=s \quad\text{and    }\quad\phi_2(s)=1-\gamma s \ln{s}.
\]
\end{lemma}

We are now in a position to specify the domains of the self-adjoint extensions of the operator $L_\ell$.
For this aim we will need the following notion, encoding the boundary terms, which we introduce just for the case of the operators $L_\ell$. 
\begin{definition}
     The Lagrange parentheses between two functions $u,v$ are
\[
[u,v]=u \overline{v}'- u'\overline{v}.
\]
\end{definition}

We have the following theorem characterizing self-adjoint extensions of $L_\ell
$.

\begin{proposition}
\label{prop: ext L}
        Let $\ell$ be a leaf of the foliation $\fol$. Then,
        \begin{itemize}
    \item if $\ell$ contains only one limit circle endpoint corresponding to zero, then the domains of self-adjointness of $L_\ell$ are parametrized by $(\alpha,\beta)\in \R^2\setminus\{(0,0)\}$ and are given by
    \begin{equation}
    \label{sl ext}
    D_{\alpha,\beta}(L_\ell)=\left\{u\in D_{\max}(L_\ell) \, : \, \alpha[u,\phi_1](0)+\beta[u,\phi_2](0)=0\right\};
\end{equation}
        \item if $\ell\sim (0,a_1)$, with both $0$ and $a_1$ in the limit circle case, then the domains of self-adjointness of $L_\ell$ are parametrized by two matrices $M,N\in M_2(\mathbb{C})$, such that $$\operatorname{rank}(M:N)=2,\, M EM^\ast=NEN^\ast,\quad\text{where}\quad E=\begin{pmatrix}
0 &-1\\
1 & 0
\end{pmatrix}$$ and are given by
\begin{equation}
\label{sr ext lclc}
    \begin{split}
        D_{M,N}(L_\ell)=&\left\{u\in D_{\max}(L_\ell) \, : \, M\binom{[u, \phi_1](0)}{[u,\phi_2](0)}+N\binom{[u, \phi_3](a_1)}{[u,\phi_4](a_1)}=\binom{0}{0}\right\}
    \end{split}
\end{equation}
\end{itemize}
Conversely, if we restrict the adjoint of $L_\ell$ to a domain of the form as above, we get a self-adjoint extension of $L_\ell$. In particular, once chosen a domain of that form on each leaf, their direct integral defines a self-adjoint extension of $\mathcal{L}$ on the whole foliation, which gives rise to disjoint dynamics.
\end{proposition}

\begin{proof}   
For $\phi_1$ and $\phi_2$ as in Lemma \ref{dmaxx ns} (proper node or saddle), we have $[\phi_1,\phi_2](0)=1$, while for $\phi_1$ and $\phi_2$ as in Lemma \ref{dmax focus} (star node or focus), we have $[\phi_1,\phi_2](0)=2\mu$. Hence, up to replacing $\phi_i$ with $\frac{\phi_i}{\sqrt{2\mu}}$ in the latter case, we can apply \parencite[Theorem 10.4.5]{zettlSturmLiouville2012} for the case of a leaf with only one limit circle, and \parencite[Proposition 10.4.2]{zettlSturmLiouville2012} for the case of a leaf with two limit circle endpoints.
\end{proof}

We now state a theorem regarding self-adjoint extensions that lead to disjoint dynamics for the operator $\bDelta$ on the foliation, defined in \eqref{direct integrals}.

\begin{theorem}
    The whole family of disjoint self-adjoint extensions $\tilde\bDelta$ of the operator $\bDelta$ with domain $\int_\fol^\oplus \mathcal{C}^\infty_0 (\ell)\,d\nu(\ell)$ is given by
    \begin{equation}
        \tilde\bDelta = \int^\oplus \tilde\Delta_\ell\,d\nu(\ell),
    \end{equation}
    where $\tilde\Delta_\ell$ is a self-adjoint extension of $\Delta_\ell$.

    Let $\lambda_0$ be given by Definition~\ref{def:lambda0}, and define the vector $v_0=(v_0^1,v_0^2)$ associated to a point $p$, corresponding to 0, as follows
    \begin{itemize}
    \item if 0 is a proper node or a saddle, then
    \begin{equation*}
        v_0=\left.
        \begin{pmatrix}
            \frac{1}{2\mu}s^{1-2\mu}\left[\left(1-\frac{\gamma}{1-2\mu}s\right)s^{-\frac{1}{2}+\mu+\frac{1}{2\lambda_0}}u(s)\right]'\\[.3em]
           {s^{-\frac{1}{2}+\mu+\frac{1}{2\lambda_0}}u(s)} 
        \end{pmatrix}
        \right|_{s=0^+}
    \end{equation*}
    \item if 0 is a star node or a focus, then \begin{equation*}
   v_0=
        \left.
        \begin{pmatrix}
        su(s)\\[.3em]
        \left[s\left(1+\gamma s\ln s\right)u(s)\right]'
   \end{pmatrix}
   \right|_{s=0^+}
   \end{equation*}

\end{itemize}
  One analogously defines $v_1$ for the endpoint $a_1$.
 
Then all self-adjoint extensions of $\Delta_\ell$ are characterised as follows:
\begin{itemize}
    \item If $\ell$ is a leaf with  no limit circle endpoint, then $\Delta_\ell$ is essentially self-adjoint and $D(\tilde {\Delta}_\ell)=D_{\min}(\Delta_\ell)$. 
    \item  If $\ell$ is a leaf with only one limit circle endpoint (say $0$), then there exists $\gamma_\ell\in (-\infty,+\infty]$ such that
           \begin{equation}
            \label{eq:domain-disjoint-half-line}
            D(\tilde\Delta_\ell) = \left\{ u \in D_{\max}(\Delta_\ell) \mid v_0^1= \gamma_\ell v_0^2 \right\}.
        \end{equation}
             
     \item If $\ell\sim (0,a_1)$, with both $0$ and $a_1$ in the limit circle case, then the domains of self-adjointness of $\Delta_\ell$ are parametrized by two matrices $M,N\in M_2(\mathbb{C})$, such that $$\operatorname{rank}(M:N)=2,\, M EM^\ast=NEN^\ast,\quad\text{where}\quad E=\begin{pmatrix}
0 &-1\\
1 & 0
\end{pmatrix}$$ and given by
\begin{equation}
\label{sr ext lclc}
    \begin{split}
        D_{M,N}=&\left\{u\in D_{\max} \, : \, M\,v_0+N\,v_1=\binom{0}{0}\right\}.
    \end{split}
\end{equation}
\end{itemize}

Conversely, if we restrict the adjoint of $\Delta_\ell$ to a domain of the form as above, then we get a self-adjoint extension of $\Delta_\ell$. In particular, once chosen a domain of that form on each leaf, their direct integral defines a self-adjoint extension of $\bDelta$ on the whole foliation, which gives rise to disjoint dynamics.

\end{theorem}

\begin{proof}

The case where both enpoints of $\ell$ are limit point is clear. We thus focus on the case where $\ell$ has at least one limit circle endpoint.

Since the operator $\Delta_\ell$ is unitarily equivalent to $L_\ell$, via the transformation $T:L^2(\ell,d\mu_\ell)\to L^2(\ell,ds)$, we can find self-adjoint extensions of $\Delta_\ell$, and then of $\bDelta$, by transforming all the domains back via $T^{-1}$. We do all computations for the case in which the zero endpoint is limit circle. The $a_1$ case is equivalent.

Performing some simple computations, using that $b=(\log\mu_\ell)'$ (see Section \ref{local b}), one finds $T^{-1}(u)=u(s)s^{-1 / 2\lambda_0} (1+ O(s))$. 

Let now $u$ be a general function in $D_{\max}(L_\ell)$. We know that there exists $\tilde{u}\in D_{\min}(L_\ell)$ and $A,B\in\R$ such that, for a proper node or a saddle, 
\[
u(s)= A\left(s^{\frac{1}{2}-\mu}+\frac{\gamma}{1-2\mu}s^{\frac{3}{2}-\mu}\right)+Bs^{\frac{1}{2}+\mu}+\tilde{u}(s),
\]
  while, for a focus or a star node, 
  \[
  u(s)=As+B(1-\gamma s \ln{s})+\tilde{u}(s).
  \] 
From this we can derive the expression of a generic function $T^{-1}(u)\in D_{\max}(\Delta_\ell)$.
Recall that 
$$\beta=\frac{1-2\lambda_0}{4\lambda_0^2},\quad \text{so that}\quad\mu=\frac{\abs{\lambda_0-1}}{2\abs{\lambda_0}}.$$ 
In particular, if $\lambda_0>0$, we get $\mu=\frac{1-\lambda_0}{2\lambda_0}$, while if $\lambda_0<0$, we get $\mu=\frac{\lambda_0-1}{2\lambda_0}$. This leads to
\begin{equation}
    \label{eq:u-dmax1}
     T^{-1}(u)=\begin{cases}
    A\left(s^{\frac{\lambda_0-1}{\lambda_0}}-\frac{\gamma}{1-2\mu}s^{\frac{2\lambda_0-1}{\lambda_0}}\right)+B+s^{-\frac{1}{2\lambda_0}}\tilde{u} \quad\text{if }\lambda_0>0,\\[.3em]
    A\left(1-\frac{\gamma}{1-2\mu}s\right)+Bs^{\frac{\lambda_0-1}{\lambda_0}}+s^{-\frac{1}{2\lambda_0}}\tilde{u} \quad\text{if }\lambda_0<0
 \end{cases}
\end{equation}

for a proper node or a saddle, and
\begin{equation}
    \label{eq:u-dmax2}
    T^{-1}(u)=A+B\left(\frac{1}{s}-\gamma\ln{s}\right)
\end{equation}
for a star node or a focus.

Now we use Proposition \ref{prop: ext L} and the fact that the Lagrange parenthesis $[\phi_i,u]$ are
\[
[\phi_1,u](0)=B
\qquad\text{and}\qquad
[\phi_2,u](0)=A,
\]
to write the domains of self-adjointness of $L_\ell$, and then of $\Delta$, by giving some constraints on the coefficients $A$ and $B$. The result follows by observing that, by \eqref{eq:u-dmax1} and \eqref{eq:u-dmax2}, one has $v_0 = (B,A)^\top$.

\end{proof}

\subsection{The case of half-line leaves via  Von Neumann's theory of self-adjoint extensions}

In this section we analyse the  problem of finding all self-adjoint extensions of
the symmetric operator $L_\ell$ again,
with an alternative method: the von Neumann theory (see e.g. \cite{reedFourier2007} or \cite{akhiezer}). Unlike Sturm-Liouville theory, which defines the domain of self-adjointness by restricting the maximal domain, von Neumann's theory constructs self-adjoint extensions by extending the minimal domain. Actually, such a method yields the functions needed to extend the minimal domain and to define self-adjoint extensions. However, it works globally on the whole leaf, and for this reason, in our setting, the computations {prove less cumbersome} only in the case of a leaf equivalent to a half-line.

Let us then consider a leaf isometric to $(0,\infty)$ with a characteristic point $p$ corresponding to the coordinate $0$ in the limit circle case (the limit point case is trivial). Suppose furthermore that 
$\zeta(s)$ is bounded on $\ell$. As proved in Lemma \ref{lem:potential}, the operator $-\Delta_\ell$ is unitarily equivalent to the operator on $L^2(\ell,ds)$ given by
\begin{equation}
   \label{ellell}
L_\ell=-\partial_s^2+\frac{\beta}{s^2}+\frac{\zeta(s)}{s},
\end{equation}
where for the case of a focus we have that $\beta$ is equal to zero.

We now state the theorem for self-adjoint extensions in the case introduced above. For an exhaustive review on Whittaker functions see \cite{richard:hal-01679505}.

\begin{theorem}
    Let $\ell\sim (0,\infty)$ with $0$ a characteristic point in the limit circle case. Suppose that the function $\zeta$ is bounded on $\ell$. 
  Then the operator $L_\ell$, defined on the domain
    $\mathcal{C}^\infty_0 ([0, + \infty))$ and acting as
    in \eqref{ellell}, admits a family
    of self-adjoint extensions in
    $L^2 (\ell,ds)$, denoted by $L_{\ell, \theta}, \ \theta \in [0, 2\pi)$, each acting as $L_\ell$ and defined on the domain 
    \begin{equation}
D_\theta=D_0+span_\mathbb{C}\{W_++e^{i\theta}W_-\}, 
 \qquad \theta\in[0,2\pi). \label{Dteta}
\end{equation}
Here $W_+$ and $W_-$ are two elements of the family of the Whittaker functions. More precisely, letting $\gamma = \zeta(0)$ and denoting by $W_{a,b}(s)$ the Whittaker function of parameters $a,b\in\mathbb{R}$, we have
\begin{itemize}
    \item if 0 is a proper node or a saddle, then
    \begin{equation}
        \label{eq:whittaker-node-saddle}
      W_+(s)= W_{-\frac{1+i}{2\sqrt{2}}\gamma,\sqrt{\beta-\frac{1}{4}}}(\sqrt{2}(1-i)s)\quad \text{and}\quad
       W_-(s)=W_{\frac{-1+i}{2\sqrt{2}}\gamma,\sqrt{\beta-\frac{1}{4}}}(\sqrt{2}(1+i)s),
    \end{equation}
    \item if 0 is a star node or a focus, then
    \begin{equation*}
        W_+(s)= W_{-\frac{1+i}{2\sqrt{2}},\frac{1}{2}}(\sqrt{2}(1-i)s)\quad \text{and}\quad
       W_-(s)=W_{-\frac{1-i}{2\sqrt{2}}, \frac{1}{2}}(\sqrt{2}(1+i)s).
    \end{equation*}
\end{itemize}
\end{theorem}

\begin{proof}
We start by observing that, by Kato-Rellich Theorem, for the purpose of determining the self-adjoint extensions of $L_\ell$ we can assume w.l.o.g.~that $\zeta(s)=\gamma$ for all $s>0$.

To construct all self-adjoint extensions of $L_\ell$
defined on $\mathcal{C}^\infty_0 (\ell)$
we apply Von Neumann's theory. To this aim we 
find the solutions of the so called deficiency equations $L_\ell^* u = \pm i u$, and use them to extend the minimal domain of $L_\ell$. 

 Repeated integration by parts shows that the action of $L_\ell^*$ is the same as that of $L_\ell$ in \eqref{ellell}. Thus the deficiency equations for $L_\ell$ read 
\begin{equation}
    \label{whit}
-u''(s)+\frac{\beta u(s)}{s^2}+\frac{\gamma u(s)}{s}=\pm i u(s).
\end{equation}
Performing the change of 
variable $z=2\sqrt{\mp i}s$, with the choice
$\sqrt{\mp i} = \frac{1\mp i}{\sqrt 2}$ that
guarantees $\Re z > 0$, one can rewrite eq. \eqref{whit} as
\[
v''(z)+\left(-\frac{1}{4}+ \frac{\lambda}{2z}-\frac{\beta}{z^2}\right)v(z) = 0,
\]
which are two Whittaker equations 
 (see Sec. 9.22-9.23 in \cite{gradstein})
with parameters $\lambda_\pm =\frac{(\mp i)^{\frac{3}{2}}}2\gamma = \frac{-1\pm i}{2\sqrt 2} \gamma$ and $\mu$ defined by $\frac{1}{4}-\mu^2=\beta$.  By the general theory of the Whittaker equations, one finds that for
every such equation  the only solution which
is square-integrabile is the Whittaker function
$W_{\lambda, \mu},$ with $\mu \in (-1, 1)$
which is guaranteed by the fact that $\beta\in (-\frac{1}{4},\frac{3}{4})$.
Then, both deficiency subspaces $K_\pm = \ker(L_\ell^*\mp i)$ have dimension one and are generated by the Whittaker functions $W_{\lambda_\pm,\mu}$. One finally has
\begin{equation}
        K_+=\ker (L_\ell^\ast-i)=\operatorname{span}\{W_+\}
        \qquad\text{and}\qquad
        K_-=\ker (L_\ell^\ast+i)=\operatorname{span}\{W_-\}.
    \end{equation}

Furthermore, from $\overline{W}_-=W_+$ it follows
$\| W_+\|_{L^2 (\R^+)} = \| W_-\|_{L^2 (\R^+)}$, therefore every unitary map $U$ from the space $K_+$ and
the space $K_-$
is characterized by a parameter $\theta \in [0, 2 \pi)$ such that $UK_+ = e^{i \theta} K_-.$ The
self-adjointness domains in \eqref{Dteta} follows directly from 
von Neumann's theory of self-adjoint extensions and the proof is complete for the case of a proper node or a saddle.
\end{proof}

\begin{remark}
    The self-adjoint extensions of $L_\ell$ are unitarily equivalent to the self-adjoint extensions of $\Delta_\ell$ whose domains are given in \eqref{eq:domain-disjoint-half-line}. 
    This can be directly checked by an asymptotic expansion of the Whittaker functions $W_+$ and $W_-$.

    More precisely, in the case of a star node or a focus, the expression of the Whittaker functions is
    $$
    \begin{aligned}
W_{\lambda_\pm, \frac{1}{2}}(z)= & -\frac{ z e^{-\frac{z}{2}}}{\Gamma\left(-\lambda_\pm\right) \Gamma\left(1-\lambda_\pm\right)} \\
& \times\left\{\sum_{k=0}^{\infty} \frac{\Gamma\left(1+k-\lambda\right)}{k!(1+k)!} z^k\left[\Psi(k+1)+\Psi(k+2)-\Psi\left(k-\lambda+1\right)-\ln z\right]\right. \\
& \left.-\frac{1}{z} \left(\Gamma(1) \Gamma(-\lambda)\right)\right\},
\end{aligned}
    $$
     where $\Psi$ is the Euler's $\Psi$ function (see Sec. 8.36 in \cite{gradstein}), which one can check to have the same leading terms as in Lemma \ref{dmax focus}.
    
     In the case of a node or saddle characteristic point, instead, there exist $\xi,\eta\in \mathbb{C}$ such that
    \begin{equation*}
        \begin{split}
        W_+(s) +e^{i\theta}W_-(s)
        &= (\xi+e^{i\theta}\bar\xi)\left(s^{\frac{1}{2}-\mu}+\frac{\gamma}{1-2\mu} s^{\frac{3}{2}-\mu}\right)+ (\eta+e^{i\theta}\bar\eta)s^{\frac{1}{2}+\mu}- +O(s^{\frac{3}{2}+\mu})\\
        &= (\xi+e^{i\theta}\bar\xi)\phi_1(s)+ (\eta+e^{i\theta}\bar\eta)\phi_2(s) +O(s^{\frac{3}{2}+\mu}).
        \end{split}
    \end{equation*}
    Here, $\phi_1,\phi_2$ are the functions in $D_{\max}(L_\ell)$ introduced in Proposition~\ref{prop: ext L}. 
    Observe that we have
    \begin{equation}
        \alpha[W_+ +e^{i\theta}W_-,\phi_1](0) - [W_+ +e^{i\theta}W_-,\phi_2](0) = 0, \qquad \alpha = \frac{\eta+e^{i\theta}\bar\eta}{\xi+e^{i\theta}\bar\xi}.
    \end{equation}
    Since, as one can easily check,  $\alpha$ is real, the self-adjoint extensions obtained in this section coincide
    with those obtained in Proposition~\ref{prop: ext L}.

    A similar result holds in the star node or focus case, with $\phi_1,\phi_2$ given in Lemma~\ref{dmax focus}
\end{remark}

\appendix
\section{The surface measure and its disintegration along the leaves}
\label{sec: disi}

In this section we present some remarks about the surface measure on $S$. We start by proving that the measure $\mu$ goes to zero at characteristic points.
\begin{proposition}
\label{prop: mu va a zero}
    Let $S$ be a surface embedded in a 3-dimensional contact sub-Riemannian manifold $M$. Then the surface measure $\mu$ defined on $S\setminus C(S)$ by \eqref{hor unit norm}, extends by continuity to zero at characteristic points.
\end{proposition}
\begin{proof}    
    Let $p\in C(S)$ and assume that $S$ is described, locally at $p$, as the zero locus of a smooth function $u$, with $du\ne0$. Let $X_1,X_2,X_0$ be a local frame with $\{X_1,X_2\}$ an orthonormal frame for the distribution and $X_0$ the Reeb vector field. Then, by \eqref{hor unit norm}, outside $p$ we have
    \[
   \mu= -\frac{X_2 u}{\norm{\nabla_H u}} \ dX_1\wedge dX_0+\frac{X_2 u}{\norm{\nabla_H u}} \ dX_2\wedge dX_0,  
    \]
    where $\nabla_H\,u$ is the horizontal gradient $\nabla_H\,u=(X_1\,u,X_2\,u)$.
    
    Since $du\ne0$, we can assume that $X_0\,u\equiv1$. Using the fact that on $TM$ it holds $X_1u \,dX_1+X_2u \,dX_2+X_0u \,dX_0=0$, we have
    \[
    \mu=\norm{\nabla_{H}u}dX_1 \wedge dX_2.
    \]
  In particular one finds that this goes to zero at characteristic points, because so does the horizontal gradient.
\end{proof}
 
We now study of the disintegration of $\mu$. 

Given a complete vector field $X$ and a point $p_0$, the flow of $X$ starting at $p_0$ defines a map
\begin{equation}
    \phi(t) = e^{tX}(p_0).
\end{equation}
We say that $I\subset \R$ is an \emph{embedding interval} if $0\in I$ and $\phi:I\to \phi(I)$ is an embedding\footnote{That is, $d\phi$ is injective and the topology induced by $\phi$ coincides with the restriction of the topology of $S$ to $\phi(I)$.}. 
When $X(p_0)\neq 0$, we have three cases for the orbit of $X$ passing through $p_0$ :
\begin{enumerate}
    \item [(i)] The orbit is $T$-periodic. In this case, $(-T/2,T/2)$ is an embedding interval.
    \item [(ii)] The orbit is recurrent\footnote{By recurrent here we mean that the orbit contains at least one point $p$ such that for any neighbourhood $U$ of $p$ and any time $t\in \R$, there exists $t^\ast>t$ such that $e^{t^\ast X}(p)\in U$.} but non-periodic. In this case, any finite interval $(-a,a)$, $a>0$, is an embedding interval. Notice that $\mathbb{R}$ is not an embedding interval.
    \item [(iii)] The orbit is non-recurrent, in which case $\mathbb{R}$ is an embedding interval.
\end{enumerate}

We have the following.

\begin{lemma}[Almost global rectification of a vector field]
\label{lem: rect}
    Let $X$ be a smooth complete vector field on a 2-dimensional smooth manifold $S$ and let $p_0\in S$ be a point such that $X(p_0)\ne 0$. 
    For every embedding interval $I\subset \R$ there exists a smooth function $\eta:I\to (0,\infty)$ and a local chart $\phi:U\to S$ defined on
    \[
    U=\{(t,s)\in \R^2: t\in I \text{ and } \abs{s}<\eta(t)\},
    \]
    and such that $\phi_\ast \partial_t=X$. 
\end{lemma}
\begin{remark}
   By construction, if the embedding interval is bounded, then $\eta$ can be chosen to be constant. When $I=\R$, and $e^{\R X}(p_0)$ connects two foci, the open set $U$ can be chosen to be fusiform.
\end{remark}

\begin{proof}
We show that the statement holds on embedding intervals of the form $(-\delta_0,b)$ by constructing a local chart on
\begin{equation}
    \label{eq:form-W}
    U = \{ (t,s)\in\mathbb{R}^2 : t\in (-\delta_0,b)\text{ and }|s|<\eta(t)\},
\end{equation}
for some $\delta_0>0$ arbitrarily small.
The same argument extends to the general case.

Since $X(p_0)\ne 0$, there exists a neighbourhood ${V}$ of $\bar{p_0}$ in which $X\ne 0$. 
Thanks to the fact that $I$ is an embedding interval, we can also assume that $\{e^{tX}(p_0): t\in I\}\cap V$ is connected.
Let $\gamma:(- 1, 1)\to S$, $\gamma(0)={p_0}$, be a curve inside ${V}$ and transversal to $X$. 
Then, there exists $W_0= (-s_0,s_0)\times (-\delta_0, \delta_0)$ such that $\Phi:(s,t)\in W_0\to e^{tX}(\gamma(s))\in S$ is a diffeomorphism from $W_0$ onto its image. This is possible since $D\Phi(t,s)$ is invertible near $(0, 0)$. Indeed, $X(\gamma(s))$ and $\dot\gamma(s)$ are linearly independent for $|s|<1$, and

\[
\frac{\partial \Phi}{\partial t}(t,s)=X(\Phi(t,s))=(e^{tX})_\ast X(\gamma(s)),
\qquad
\frac{\partial \Phi}{\partial s}(t,s)=(e^{tX})_\ast (\dot\gamma(s)).
\]

Iterating this procedure, we construct a sequence of points $(p_n)_{n\in\mathbb{N}}$, open sets $(W_n)_{n\in\mathbb{N}}$, and diffeomorphisms $\Phi_n:W_n\to \Phi_n(W_n)$ such that 
\begin{itemize}
    \item $W_n = (-s_n,s_n)\times (t_n-\delta_n,t_n+\delta_n)$ for $s_n,\delta_n>0$, where $t_n = \frac{1}{2}\sum_{j=0}^{n-1}\delta_j$;
    \item $p_n = \operatorname{exp}(t_n X)(p_0)$ for any $n\ge 1$;
    \item $\Phi_n(s,t)=\operatorname{exp}\left(t  X\right)(\gamma(s))$;
    \item $\Phi_n(W_n\setminus W_{n-1}) \cap \bigcup_{m=0}^{n-1} \Phi_m(W_m)= \emptyset$.
\end{itemize}
Since, by construction, $\Phi_n$ coincides with $\Phi_{n+1}$ on $W_n\cap W_{n+1}$ this allows to define a map $\Phi : W \to \Phi(W)$, where  $W=\bigcup_{n=0}^{+\infty} W_n$. The last condition above, which we can require thanks to the fact that $I$ is an embedding interval, implies that $\Phi$ is a diffeomorphism.

To conclude the proof it suffices to show that $W$ contains an open set of the form \eqref{eq:form-W}. Observe that $(t_n)_n$ is an increasing sequence and that $t_n<b$ for any $n$. Hence, up to eventually reducing $W$, we have
\begin{equation}
    W = \{ (t,s)\in\mathbb{R}^2 : t\in (-\delta_0,\bar t)\text{ and }|s|<h(t)\},
\end{equation}
where $h(t) = s_{n}$ for $t\in [t_n,t_{n+1}]$ and $\bar t =\lim_n t_n\le b$. Since it is always possible to find a smooth function $\eta$ such that $\eta(t)\le h(t)$, we are left to prove that we can perform the above construction in such a way that $\bar  t = b$. 

Let us assume that $\bar t<b$ and let us show that we can extend the above construction to yield a final time strictly bigger than $\bar t$. 
Let $\bar p:=e^{\bar t X}(p_0)$.
Since $X(\bar p)\neq 0$, there exists a curve $\tilde \gamma:(-1,1)\to S$ transversal to $X$, such that $\tilde\gamma(0)= \bar p$. Hence, we can construct as above a local diffeomorphism $\Psi: A \to \Psi(A)$ where $A = (-s',s')\times (-\delta',\delta')$. In particular, there exists $n_0\in \mathbb{N}$ such that $\bar t-t_{n_0} < \delta'$ and hence, up to reducing $s_{n_0}$, we have $\exp(t_{n_0} X)\gamma(s)\in \Psi(A)$ for any $|s|<s_{n_0}$. This shows that we can modify the construction above, replacing $\delta_{n_0}$ with $\bar t + \delta' - t_{n_0}$. Here we need to possibly reduce again $s_{n_0}$ in order to have everything well defined when $t\in (-\delta_{n_0},0)$) In particular, $t_{n_0+1}>\bar t$.
\end{proof}

By exploiting the coordinates given by the previous rectification theorem, we obtain the following.

\begin{lemma}[Disintegration]
\label{lem: dis}
Let $\mu$ be a smooth measure on $S$ and let $U$ be the open set given by Lemma \ref{lem: rect} for some smooth complete vector field $X$ such that $X(p_0)\neq 0$. Then, for any given smooth function $\nu$, $\mu=\mu_s(t)dt\wedge  \nu(s) ds$ on $U$, where $(s,t)\in U\mapsto \mu_s(t)$ is a smooth function.
\end{lemma}
The measure $\nu(s)ds$ in the statement is an arbitrary 1-dimensional measure that is defined transversally to the leaves. In the system of coordinates given by the Lemma, the function $\nu(s)$ can be taken equal to one, however one has to consider that different $\nu(s)$ give rise to different $\mu_s(t)$. 
Nevertheless, for fixed $s$, two different decompositions $\mu_s(t)$, corresponding to different $\nu(s)$, differ only by a non zero constant.

Thanks to the above lemma, we can actually define an intrinsic notion of measure on $\ell$, coming from the disintegration of $\mu$. We explain this in the following Proposition.

\begin{proposition}
    Let $S$ be a surface embedded in a 3-dimensional contact sub-Riemannian manifold $M$ and let $\ell\in \fol$ be a leaf of the characteristic foliation induced on $S$.
    Then an intrinsic measure $\mu_\ell$ on $\ell$ is defined up to a constant, and it holds that $\Delta_\ell = \operatorname{div}_{\mu_\ell}\nabla_\ell$.
\end{proposition}

\begin{proof}
Let $\ell$ be a leaf of $\fol$. To show the existence of the measure $\mu_\ell$, we proceed by covering the leaf into at most countably many patches, in such a way that on every patch, the measure $\mu$ can be disintegrated using Lemma \ref{lem: dis}. This is possible because, locally, one can always define a characteristic vector field, which is non zero at non characteristic points.

 Once we have the disintegrated measure on each patch, we show that the disintegrated measure on a fixed leaf agrees on the intersections of patches. This is possible because one can rectify $X$ on the intersection of the two neighbourhood and see that the two disintegrated measure must differ by a non-zero multiplicative constant (see \cite{AVILA_VIANA_WILKINSON_2022}). In particular, we can adjust the constant sequentially so that we get a well-defined measure on each leaf, which is unique up to multiplicative constant. 

It remains to show that $\Delta_\ell = \operatorname{div}_{\mu_\ell}\nabla_\ell$. 
Without loss of generality, we prove this fact under the hypothesis of Lemma \ref{lem: dis}. 
We first note that, since div$_{\mu_\ell} f=$div$_{c\mu_\ell}f$ for every non vanishing constant $c$, the definition is well posed. 
Consider now the coordinates $(s,t)$ of $S$ given by Lemma \ref{lem: rect}.
For any measure $\nu$, the expression div$_\nu \nabla_\ell$ is
\[
\XS^2+\text{div}_\mu(\XS)\XS,
\]
so that in particular it is enough to show that div$_\mu(\XS)=$div$_{\mu_\ell}(\XS)$.
But this comes directly from the fact that, in $s,t$ coordinates, $\XS=\partial_t$, $\mu=f(s,t)ds\,dt$ and $\mu_\ell=f(s_\ell,t)dt$, so that
\[
\div_\mu(\partial_t)=\frac{1}{f}[\partial_t(f)]=\div_{\mu_\ell}(\partial_t).
\]
This concludes the proof.
\end{proof}

\printbibliography

@book{gradstein,
  author = {Gradshteyn, I. S. and Ryzhik, I. M.},
  description = {MR: Publications results for "MR Number=(2360010)"},
  edition = {Seventh},
  isbn = {978-0-12-373637-6; 0-12-373637-4},
  mrclass = {00A22 (33-00 65-00 65A05)},
  mrnumber = {2360010 (2008g:00005)},
  note = {Translated from the Russian, Translation edited and with a preface by Alan Jeffrey and Daniel Zwillinger, With one CD-ROM (Windows, Macintosh and UNIX)},
  pages = {xlviii+1171},
  publisher = {Elsevier/Academic Press, Amsterdam},
  title = {Table of integrals, series, and products},
  year = 2007
}

@article{vonNeumann1939,
author = {von Neumann, J.},
journal = {Compositio Mathematica},
language = {eng},
pages = {1-77},
publisher = {Johnson Reprint Corporation},
title = {On infinite direct products},
url = {http://eudml.org/doc/88704},
volume = {6},
year = {1939},
}

@article{AVILA_VIANA_WILKINSON_2022, title={Absolute continuity, Lyapunov exponents, and rigidity II: systems with compact center leaves}, 
volume={42}, 
DOI={10.1017/etds.2021.42}, 
number={2}, 
journal={Ergodic Theory and Dynamical Systems}, 
author={Avila, A. and Viana, Marcelo and Wilkinson, A.}, 
year={2022}, 
pages={437–490}}

@inproceedings{Eugenio-quantization,
author = {Agrachev, Andrei and Baranzini, Stefano and Bellini, Eugenio and Rizzi, Luca},
year = {2024},
month = {06},
pages = {},
title = {Quantitative tightness for three-dimensional contact manifolds: a sub-Riemannian approach},
doi = {10.48550/arXiv.2407.00770}
}

@article{richard:hal-01679505,
  TITLE = {{On radial Schr{\"o}dinger operators with a Coulomb potential}},
  AUTHOR = {Richard, Serge and Derezinski, Jan},
  URL = {https://hal.science/hal-01679505},
  JOURNAL = {{Ann. Henri Poincare}},
  VOLUME = {19},
  PAGES = {2869--2917},
  YEAR = {2018},
  DOI = {10.1007/s00023-018-0701-7},
  HAL_ID = {hal-01679505},
  HAL_VERSION = {v1},
}

@book{akhiezer,
  title = {Theory of Linear Operators in Hilbert Spaces},
  author = {Akhiezer, Naum Ilic and 
            Glazman, Izrail Markovich},
  year = {1993},
  publisher = {{Dover Publications Inc.}},
  address = {{New York}},
  }

@book{agrachevComprehensive2019a,
  title = {A Comprehensive Introduction to Sub-Riemannian Geometry},
  author = {Agrachev, Andrei and Barilari, Davide and Boscain, Ugo},
  year = {2019},
  series = {Cambridge Studies in Advanced Mathematics},
  publisher = {{Cambridge University Press}},
  address = {{Cambridge}},
  doi = {10.1017/9781108677325},
  collection = {Cambridge Studies in Advanced Mathematics}
}

@article{Braverman_2002,
   title={Essential self-adjointness of Schrödinger-type operators on manifolds},
   volume={57},
   ISSN={1468-4829},
   url={http://dx.doi.org/10.1070/RM2002v057n04ABEH000532},
   DOI={10.1070/rm2002v057n04abeh000532},
   number={4},
   journal={Russian Mathematical Surveys},
   publisher={Steklov Mathematical Institute},
   author={Braverman, M and Milatovic, O and Shubin, M},
   year={2002},
   month=aug, pages={641–692} }

@article{STRICHARTZ198348,
title = {Analysis of the Laplacian on the complete Riemannian manifold},
journal = {Journal of Functional Analysis},
volume = {52},
number = {1},
pages = {48-79},
year = {1983},
issn = {0022-1236},
doi = {https://doi.org/10.1016/0022-1236(83)90090-3},
url = {https://www.sciencedirect.com/science/article/pii/0022123683900903},
author = {Robert S Strichartz}
}

@article{barilariInduced2022,
  title = {On the Induced Geometry on Surfaces in {{3D}} Contact Sub-{{Riemannian}} Manifolds},
  author = {Barilari, Davide and Boscain, Ugo and Cannarsa, Daniele},
  year = {2022},
  journal = {ESAIM: Control, Optimisation and Calculus of Variations},
  volume = {28},
  pages = {9},
  issn = {1292-8119, 1262-3377},
  doi = {10.1051/cocv/2021104},
  langid = {english}
}

@article{mugnolo1, title={Self‐adjoint and Markovian extensions of infinite quantum graphs}, volume={105}, DOI={10.1112/jlms.12539}, number={2}, journal={Journal of the London Mathematical Society}, author={Kostenko, Aleksey and Mugnolo, Delio and Nicolussi, Noema}, year={2022}, pages={1262–1313}}

@book{berkolaikokuchment, place={Providence, RI}, title={Introduction to quantum graphs}, publisher={American Mathematical Society}, author={Berkolaiko, Gregory and Kuchment, Peter}, year={2013}}

@article{Kostrykin_Schrader_1999, title={Kirchhoff’s rule for quantum wires}, volume={32}, DOI={10.1088/0305-4470/32/4/006}, number={4}, journal={Journal of Physics A: Mathematical and General}, author={Kostrykin, V and Schrader, R}, year={1999}, month={Jan}, pages={595–630}}

@article{barilariStochastic2021,
  title = {Stochastic Processes on Surfaces in Three-Dimensional Contact Sub-{{Riemannian}} Manifolds},
  author = {Barilari, Davide and Boscain, Ugo and Cannarsa, Daniele and Habermann, Karen},
  year = {2021},
  month = jul,
  journal = {Annales de l'Institut Henri Poincar\'e, Probabilit\'es et Statistiques},
  volume = {57},
  number = {3},
  eprint = {2004.13700},
  eprinttype = {arxiv},
  issn = {0246-0203},
  doi = {10.1214/20-AIHP1124},
  archiveprefix = {arXiv},
  langid = {english}
}

@book{reedFourier2007,
  title = {Fourier Analysis, Self-Adjointness},
  author = {Reed, Michael and Simon, Barry},
  year = {2007},
  series = {Methods of Modern Mathematical Physics},
  edition = {Nachdr.},
  number = {Michael Reed; Barry Simon ; 2},
  publisher = {{Acad. Press}},
  address = {{New York, NY}},
  isbn = {978-0-12-585002-5},
  langid = {english}
}

@book{reedMethods1980,
  title = {Methods of Modern Mathematical Physics},
  author = {Reed, Michael and Simon, Barry},
  year = {1980},
  edition = {Rev. and enl. ed},
  publisher = {{Academic Press}},
  address = {{New York}},
  isbn = {978-0-12-585050-6},
  lccn = {QC20.7.F84 R43 1980}
}

@book{zettlSturmLiouville2012,
  title = {Sturm-{{Liouville Theory}}},
  author = {Zettl, Anton},
  year = {2012},
  publisher = {{American Mathematical Society}},
  address = {{S.l.}},
  isbn = {978-0-8218-5267-5},
  langid = {english}
}

@book{BennequinEQPfaff,
  title = {Entrelacements et équations de Pfaff},
  author = {Daniel Bennequin},
  year = {1983},
  journal = {Third Schnepfenried geometry conference},
  volume = {1 (Schnepfenried, 1982). Vol 107},
  pages = {87--161},
publisher = {{Astérisque. Soc. Math. France}},
  address = {{Paris}}
}

@article{GirouxConvexité91,
  title = {Convexité en topologie de contact},
  author = {Emmanuel Giroux},
  year = {1991},
  journal = {Comment. Math. Helv. 66.4},
  pages = {637-–677},
  issn = {0010-2571},
  doi = {10.1007/BF02566670},
 url= {https://doi.org/10.1007/BF02566670}
}

@article{GirouxBifurcat00,
  title = {Structures de contact en dimension trois et
bifurcations des feuilletages de surfaces},
  author = {Emmanuel Giroux},
  year = {2000},
  journal = {Invent. Math. Helv. 141.3},
  pages = {615-689},
  issn = {0020-9910},
  doi = {10.1007/s002220000082},
 url= {https://doi.org/10.1007/s002220000082}
}

@article{DANIELLI2007292,
title = {Sub-Riemannian calculus on hypersurfaces in Carnot groups},
journal = {Advances in Mathematics},
volume = {215},
number = {1},
pages = {292-378},
year = {2007},
issn = {0001-8708},
doi = {https://doi.org/10.1016/j.aim.2007.04.004},
url = {https://www.sciencedirect.com/science/article/pii/S0001870807001181},
author = {D. Danielli and N. Garofalo and D.M. Nhieu}
}

@article{DanielliMean12,
author = {Danielli, Donatella and Garofalo, Nicola and Nhieu, Duy-Minh},
year = {2012},
month = {03},
pages = {},
title = {Integrability of the sub-Riemannian mean curvature of surfaces in the Heisenberg group},
volume = {140},
journal = {Proceedings of the American Mathematical Society},
doi = {10.1090/S0002-9939-2011-11058-X}
}

@article{BaloghCorrectionHeisenberg20,
author={Balogh, Zoltán M.
and Tyson, Jeremy T.
and  Vecchi, Eugenio},
year={ 2020},
title= {Correction to: Intrinsic curvature of curves and surfaces and a Gauss–Bonnet theorem in the Heisenberg group},
journal={ Mathematische Zeitschrift},
issn={ 1432-1823},
url={https://doi.org/10.1007/s00209-019-02234-8},
doi= {10.1007/s00209-019-02234-8}
}

@article{BaloghHeisen17,
author={Balogh, Zoltán M.
and Tyson, Jeremy T.
and  Vecchi, Eugenio},
year={ 2020},
title= {Intrinsic curvature of curves and surfaces and a Gauss–Bonnet theorem in the Heisenberg group},
journal={ Mathematische Zeitschrift 296.1-2},
issn={1432-1823},
url={ https://doi.org/10.1007/s00209-016-1815-6},
doi={10.1007/s00209-016-1815-6}
}

@article{VelosoGauss-Bonnet20,
author = {Veloso, Jose},
year = {2020},
month = {02},
pages = {},
title = {Limit of Gaussian and normal curvatures of surfaces in Riemannian approximation scheme for sub-Riemannian three dimensional manifolds and Gauss-Bonnet theorem}
}

@article{walsh,
    author = {John B. Walsh},
    title = {A diffusion with a discontinuous local time},
    journal = {In Temps locaux, number 52-53 in Astérisque} ,
    year = 1978,
   pages = {37-45},
publisher = {Société mathématique de France}
}

@book{geiges-contact, place={Cambridge}, series={Cambridge Studies in Advanced Mathematics}, title={An Introduction to Contact Topology}, publisher={Cambridge University Press}, author={Geiges, Hansjörg}, year={2008}, collection={Cambridge Studies in Advanced Mathematics}}

@article{Barilari_2023,
   title={Intrinsic sub-Laplacian for hypersurface in a contact sub-Riemannian manifold},
   volume={31},
   ISSN={1420-9004},
   url={http://dx.doi.org/10.1007/s00030-023-00891-7},
   DOI={10.1007/s00030-023-00891-7},
   number={1},
   journal={Nonlinear Differential Equations and Applications NoDEA},
   publisher={Springer Science and Business Media LLC},
   author={Barilari, Davide and Habermann, Karen},
   year={2023},
   month=oct }

\end{document}